\documentclass{amsproc}

\input{macros.tex}

\begin{document}

\setcounter{page}{203}

\title{Exactly solving the KPZ equation}


\author{Ivan Corwin}
\address{Columbia University, Department of Mathematics, 2990 Broadway, New York, NY 10027, USA}
\email{ivan.corwin@gmail.com}



\date{\today}

\begin{abstract}
We present a complete proof of the exact formula for the one-point distribution for the narrow-wedge Hopf-Cole solution to the Kardar-Parisi-Zhang (KPZ) equation. This presentation is intended to be self-contained so no previous knowledge about stochastic PDEs, or exactly solvable systems is presumed.
\end{abstract}

\maketitle


\tableofcontents

\section{Introduction}
\subsection{Preface}
The study of the Kardar-Parisi-Zhang (KPZ) equation and its associated universality class has spurred and paralleled the development of an immense amount of deep and fascinating mathematics and physics. The core study of the KPZ equation and universality class warrants a detailed treatment, and so do many of the related fields developed in parallel (e.g. the study of non-linear stochastic PDES, or of integrable probabilistic systems). This  article purports to do much less -- the only goal being to present the main steps of a proof of the exact formula for the distribution of a particular solution to the KPZ equation. In so doing, we will touch on a number of the aforementioned related developments. Exercising great self-control, we will forego diversions into these areas.

Before delving into methods, let us quickly (and informally) answer two questions -- what is the KPZ equation and what does it mean to exactly solve it?

The KPZ equation was introduced in 1986 by Kardar, Parisi and Zhang \cite{KPZ} as a model for random interface growth. In one-spatial dimension (sometimes also called $(1+1)$-dimension to emphasize that time is one dimension too) it describes the evolution of a function $h(t,x)$ recording the height of an interface at time $t\geq 0$ above position $x\in \R$. The equation for $h(t,x)$ is the following non-linear stochastic partial differential equation (SPDE)\index{Kardar-Parisi-Zhang (KPZ)!equation}%
$$
\partial_t h(t,x) = \tfrac{1}{2} \partial_x^2 h(t,x) + \tfrac{1}{2}\big(\partial_x h(t,x)\big)^2 + \xi(t,x),
$$
which is driven by a space-time white noise $\xi$. The KPZ equation is supposed to be a representative model within a large class (called the KPZ universality class) in the sense that its long-time, large-scale behavior should be shared by many other growth models. The key characteristic of this KPZ universality class is supposed to be the presence of a local growth mechanism which involves a non-linear dependence on the local slope (here, the term $\tfrac{1}{2}\big(\partial_x h(t,x)\big)^2$), and  the presence of space-time random forcing (here, the term $\xi(t,x)$). The Laplacian provides a smoothing mechanism to compete with the roughening by the noise -- this mechanism is not believed to be particularly necessary for models to lie in the KPZ universality class.

Putting aside, for the moment, the interesting and important mathematical question of what it actually means to be a solution to this equation\footnote{Sometimes, see e.g. \cite{HairerKPZ}, people use the term ``solve the KPZ'' to mean constructing solutions. Here, we use the phrase ``exactly solve'' in place of ``solve'' to emphasize that we are interested not just in the existence of solutions but the exact formulas for their distribution.}, a physically important question regarding the KPZ equation is to understand what its solutions look like, from a statistical perspective. Since the white noise $\xi$ is random, the solution $h(t,x)$ will likewise be random. So, we seek to understand the distribution of this solution as well as its dependence on the initial data $h(0,x)$. The holy-grail is to have an exact formula for the distribution of the solution (or for certain characterizing functions, or marginals of the solution). Given such exact formulas, we might be able to exactly describe the long-time, large-scape behavior of the model (and hence predict the behavior of the entire KPZ universality class). For instance, is there a limiting deterministic growth velocity $v$ so that
\begin{equation}\label{eqdes1}
\lim_{t\to\infty}\frac{h(t,x)}{t} = v?
\end{equation}
If so, after centering by that velocity, what is the scale $\chi$ of the remaining fluctuations and what is their distribution $F(\cdot)$ so that for all $s\in \R$,
\begin{equation}\label{eqdes2}
\lim_{t\to\infty}\PP\left(\frac{h(t,x)-vt}{t^{\chi}}\leq s\right) = F(s)?
\end{equation}

It is certainly not assured (and generally not true) that an SPDE will have tractable formulas for the distribution of its solution (for instance, for the one-point marginal distribution at a fixed time and space location). However, the KPZ equation happens to be an exception as it is closely related to certain structures from integrable systems and representation theory.

\subsection{In what way is the KPZ equation exactly solvable?}
The first hint that there is a nice structure behind the KPZ equation comes from looking at the moments of its exponential (or Cole-Hopf transform)
$$
z(t,x) = e^{h(t,x)}.
$$
Formally, this transformation turns the KPZ equation into the stochastic heat equation (SHE) with multiplicative noise:\index{Stochastic heat equation (SHE)}%
$$
\partial_t z(t,x) = \tfrac{1}{2}\partial_{xx} z(t,x) + z(t,x) \xi(t,x).
$$
Writing $\EE$ to denote the expected value of a random variable with respect to the randomness induced by the noise $\xi$, notice that since $\xi$ is a centered Gaussian process, $\bar{z}(t,x):= \EE\big[z(t,x)\big]$ actually solves the deterministic heat equation $\partial_t \bar{z}(t,x) = \tfrac{1}{2}\partial_{xx} \bar{z}(t,x)$. Amazingly, higher moments
$$
\bar{z}(t,\vec{x}):= \EE\big[\prod_{i=1}^{k}z(t,x_i)\big],
$$
with $\vec{x}= (x_1,\ldots,x_k)$, also solve simple, deterministic evolution equations. This observation goes back to Kardar \cite{Kardar} and independently Molchanov \cite{Mol} wherein they argued that for any $k$,
$$
\partial_t \bar{z}(t,\vec{x}) = \big(H \bar{z}\big)(t,x),
$$
where $H$ acts in the $\vec{x}$-variables as
$$
H = \tfrac{1}{2} \sum_{i=1}^{k} \partial_{x^2_i} + \sum_{1\leq i<j\leq k} \delta(x_i-x_j=0).
$$
Here $\delta(x=0)$ is the delta functions at $0$ (so $\int f(x)\delta(x=0) = f(0)$ for a suitable class of $f$). $H$ is called the attractive, imaginary time delta Bose gas or Lieb-Liniger Hamiltonian, and it has been known since work of Lieb and Liniger from 1963 \cite{LL} that this Hamiltonian can be diagonalized explicitly using a tool known as the Bethe ansatz. \index{Bethe ansatz}%
Using those tools, one can show (see for instance \cite{BBC}) that for $x_1\leq \cdots \leq x_k$, \index{Nested contour integral}%
$$
\bar{z}(t;\vec{x}) =  \int_{\alpha_1-\I \infty}^{\alpha_1+\I\infty} \frac{dz_1}{2\pi \I} \cdots \int_{\alpha_k-\I \infty}^{\alpha_k+\I\infty} \frac{dz_k}{2\pi \I} \prod_{1\leq A<B\leq k} \frac{z_A-z_B}{z_A-z_B-1} \, \prod_{j=1}^{k} e^{\frac{t}{2} z_j^2 + x_jz_j},
$$
where we assume that $\alpha_1>\alpha_2 + 1 > \alpha_3 + 2> \cdots > \alpha_k + (k-1)$ (note that $z_i$ are complex variables of integration in the above formula on the right-hand side).

Armed with such a precise formula, one might expect that it is easy to recover the desired answers to questions \eqref{eqdes1} and \eqref{eqdes2}. There is, however, a fundamental problem -- the moments of the SHE do not characterize the distribution of the solution! Indeed, from the above formula one can show that the $k^{th}$ moment grows like exponential of $k^3$, which is far to fast to lead to a well-posed moment problem. Indeed, this is because the upper tail of $z$ is sufficiently heavy so that the moments really only tease out the nature of that tail. This phenomena is related to the concept of intermittency.

The key idea which will resolve this roadblock of an ill-posed moment problem is that we should ``be wise and discretize''. In other words, we will look for discrete regularizations of the KPZ equation which preserves the above mentioned algebraic properties, yet no longer suffers from this moment problem. As we will see in this article, $q$-TASEP provides one such option (there are others).

\subsection{Outline}
This article is split in two parts -- the first three sections are devoted to defining the KPZ equation and proving convergence of some discrete models to it. The remaining four sections then turn to the exact solvability, in particular focusing on the discrete model $q$-TASEP, as well as the semi-discrete directed polymer which arises as an intermediate limit of $q$-TASEP to the KPZ equation. In more detail, Section \ref{L1} defines what it means to solve the stochastic heat equation and constructs the delta initial data solution via chaos series. Section \ref{L3} turns to the KPZ equation and provides an explanation for why it arises as a continuum limit for many systems. Section \ref{L4} works out such a convergence result for the semi-discrete directed polymer in detail. Section \ref{L5} introduces the interacting particle system $q$-TASEP and gives its Markov duality. Section \ref{L6} uses this to give explicit contour integral formulas for certain $q$-moments of $q$-TASEP, and then describes the relation of duality and these formulas to the polymer replica method. Section \ref{L7} studies the intermittency of these moments (from the perspective of the parabolic Anderson model) and describes how to expand the contour integral formulas. Section \ref{L8} uses this expansion to gather all $q$-moments as a $q$-Laplace transform and performs the asymptotic analysis necessary to arrive at the KPZ equation statistics advertised above in the introduction. Section \ref{secspec} explains the origins of the integral formuls in Section \ref{L6} in terms of the coordinate Bethe ansatz and the spectral theory for the $q$-Boson process.  The appendices provide background on stochastic integration with respect to white noise, the Tracy-Widom distribution, and steepest descent asymptotic analysis.

\subsection{Acknowledgements}
This article is based on the author's lectures at a 2014 MSRI summer school on stochastic PDEs (and then revised for the 2017 AMS MAA Joint Meeting Short Course. The author thanks Jeffrey Kuan for writing up the handwritten notes, as well as for providing some valuable background materials, and Xuan Wu for making a close reading of the file and adding indexing to it; as well as Yuri Bahktin and James Nolin for co-organizing the school with him. The author was partially supported by the NSF through DMS-1208998, DMS-1664650, the Clay Mathematics Institute through a Clay Research Fellowship, the Institute Henri Poincare through the Poincare Chair, and the Packard Foundation through a Packard Fellowship for Science and Engineering.

\section{Mild solution to the stochastic heat equation}\label{L1}
The stochastic heat equation (SHE) with multiplicative noise looks (in differential form) like\index{Stochastic heat equation (SHE)}%
$$
\begin{cases}
\partial_t z &= \tfrac{1}{2}\partial_{xx}z + z\xi \\
z(0,x) &= z_0(x)
\end{cases}
$$
where $z:\mathbb{R}_+ \times \mathbb{R} \rightarrow \mathbb{R}$ and $z_0$ is (possibly random) initial data which is independent of the white noise $\xi$. Recall that formally $\xi$ has covariance
$$
\E\big[\xi(t,x)\xi(s,y)\big] \textrm{``}=\textrm{''} \delta_{t=s}\delta_{x=y},
$$
though this is only true in a weak, or integrated sense. See Appendix \ref{secwhitenoise} for background on $\xi$. The noise $\xi$ is constructed on a probability space $L^2(\Omega,\mathcal{F},\mathbb{P})$.

\begin{definition}
A \textit{mild solution} to the SHE satisfies for all $t>0,x\in \R$ the Duhamel form equation
$$
z(t,x) = \int_{\R} p(t,x-y)z_0(y)dy + \int_0^t \int_{\R} p(t-s,x-y)z(s,y)\xi(s,y)dyds,
$$
where $p(t,x) = \tfrac{1}{\sqrt{2\pi t}}e^{-x^2/2t}$ is the heat kernel or fundamental solution to $\partial_t p=\partial_{xx} p$, with $p(0,x)=\delta_{x=0}$. In this, we must have that $\int_0^t \int_{\R} p^2(t-s,x-y) \E\big[z^2(s,y)\big]dyds < \infty$ for the It\^{o} integrals to make sense and be finite.
\end{definition}

In this article we will be concerned only with the special delta initial data solution to the SHE. We will use a chaos series (see also Appendix \ref{secwhitenoise}) to construct the solution to the SHE with $z_0(x) = \delta_{x=0}$. \index{Chaos series}%

\begin{theorem}
Consider the class of solutions $z(t,x)$ to the mild formulation of SHE which satisfy {\rm(}i{\rm)} for all test functions $\psi,$ $\int_{\R} \psi(x)z(t,x)dx\rightarrow \int_{\R} \psi(x)\delta_{x=0}dx = \psi(0)$ as $t\searrow 0$ with $L^2(\Omega,\mathcal{F},\mathbb{P})$ convergence; {\rm(}ii{\rm)} for all $T>0$ there exists $C=C(T)>0$ such that for all $0<t\leq T$, and $x\in \R$, the bound $\E\big[z^2(t,x)\big]\leq C p^2(t,x)$ holds.

Within the above class of functions there exists a unique solution to SHE with $\delta_{x=0}$ initial data, and that solution takes the following form as a convergent {\rm(}in $L^2(\Omega,\mathcal{F},\mathbb{P})${\rm)} chaos series:
$$
z(t,x) = \sum_{k=0}^{\infty} \int_{\Delta_k(t)} \int_{\R^k} P_{k;t,x}(\vec{s},\vec{y}) d\xi^{\otimes k}(\vec{s},\vec{y}) =: \sum_{k=0}^{\infty} I_k(t,x),
$$
where $\Delta_k(t)=\{0\leq s_1 < \cdots < s_k\leq t\}$, and
$$P_{k;t,x}(\vec{s},\vec{y}) = p(t-s_k,x-y_k)p(s_k-s_{k-1}, y_k - y_{k-1}) \cdots p(s_2-s_1, y_2-y_1)p(s_1,y_1),$$
or in other words, represents the transition densities of a one-dimensional Brownian motion started at $(0,0)$ to go through the {\rm(}time,space{\rm)} points $(s_1,y_1),\ldots,(s_k,y_k)$, $(t,x)$. Also, recall from Appendix \ref{secwhitenoise} that $d\xi^{\otimes k}(\vec{s},\vec{y})$ is the multiple stochastic integral.
\end{theorem}
\begin{proof}
We proceed according to the following steps:
\begin{enumerate}
\item Show that the chaos series is convergent.
\item Prove that $\E\big[z^2(t,x)\big] \leq C p^2(t,x)$.
\item Show that the series solves SHE with $\delta_{x=0}$ initial data.
\item Argue uniqueness.
\end{enumerate}
\noindent{\bf Step (1):} Consider a random variable $X\in L^2(\Omega,\mathcal{F},\mathbb{P})$
such that
$$
X = \sum_{k=0}^{\infty} \int_{\Delta_k(\infty)} \int_{\R^k} f_k(\vec{t},\vec{x})d\xi^{\otimes k}(\vec{t},\vec{x}).
$$
for functions $f_k\in L^2(\Delta_k(\infty) \times \R^k ), k=0,1,2,\ldots$. (From Appendix \ref{secwhitenoise} we know that all $X$ have a unique decomposition into this form.)
By the multiple stochastic integral covariance isometry
$$
\E\left[ \int_{\Delta_k(\infty)}\int_{\R^k} f(\vec{t},\vec{x})d\xi^{\otimes k}(\vec{t},\vec{x}) \int_{\Delta_j(\infty)}\int_{\R^k} g(\vec{t},\vec{x})d\xi^{\otimes j}(\vec{t},\vec{x})\right] = \mathbf{1}_{j=k} \langle f,g\rangle_{L^2(\Delta_k \times \R_k)}.
$$
Thus,
$$
\E\big[X^2\big] = \sum_{k=0}^{\infty} \Vert f_k \Vert^2_{L^2(\Delta_k \times \R^k)}.
$$
This can be applied to the series for $z(t,x)$.

\begin{exercise}\label{exint}
Prove that
$$
p^2(t,y) = \frac{1}{\sqrt{2\pi t}}p(t,\sqrt{2}y) = \frac{1}{\sqrt{4\pi t}} p(t/2,y)$$
and
$$
\int_{\R} p^2(t-u,x-z)p^2(u-s,z-y)dz = \sqrt{\frac{t-s}{4\pi(t-u)(u-s)}}p^2(t-s,x-y).
$$
\end{exercise}

\begin{lemma}\label{lemIk}
$$
\E\big[I_k^2(t,x)\big] = \frac{t^{k/2}}{(4\pi)^{k/2}}\frac{\Gamma(\tfrac{1}{2})^k}{\Gamma(\tfrac{k}{2})} p^2(t,x).
$$
\end{lemma}
\begin{proof}
The multiple stochastic integral covariance isometry implies that
$$
\E\big[I_k^2(t,x)\big]  = \int_{\Delta_k(t)} \int_{\R^k} P_{k;t,x}^2(\vec{s},\vec{y})d\vec{y}d\vec{s},
$$
or more explicitly,
\begin{align*}
\E\big[I_k^2(t,x)\big] = \int_{\Delta_k(t)}d\vec{s}\int_{\R^k} d\vec{y}\; &p^2(s_1,y_1)p^2(s_2-s_1,y_2-y_1) \cdots\\
 &\quad\cdots p^2(s_k-s_{k-1},y_k-y_{k-1})p^2(t-s_k,x-y_k).
\end{align*}
Taking the integrals starting with $y_k$ down to $y_1$ and using Exercise \ref{exint},
\begin{align*}
\E\big[I_k^2(t,x)\big] &= \int_{\Delta_k(t)} d\vec{s} \;p^2(t,x) \sqrt{\frac{t}{4\pi (t-s_1)s_1}} \sqrt{\frac{t-s_1}{4\pi (t-s_2)(s_2-s_1)}} \cdots \\
&\qquad\qquad\qquad\qquad\qquad\cdots\sqrt{\frac{t-s_{k-1}}{4\pi (t-s_k)(s_k-s_{k-1})}}\\
&= \frac{t^{1/2}p^2(t,x)}{(4\pi)^{k/2}} \int_{\Delta_k(t)} d\vec{s} \frac{1}{\sqrt{s_1}} \frac{1}{\sqrt{s_2 -  s_1}}\cdots \frac{1}{\sqrt{s_k - s_{k-1}}} \frac{1}{\sqrt{t - s_k}}\,.
\end{align*}
Changing variables to factor out the $t$'s yields
$$
\E\big[I_k^2(t,x)\big] = \frac{t^{k/2}}{(4\pi)^{k/2}} p^2(t,x) \int_{\Delta_k(1)} d\vec{s} \frac{1}{\sqrt{s_1}} \frac{1}{\sqrt{s_2 -  s_1}} \cdots \frac{1}{\sqrt{s_k -  s_{k-1}}} \frac{1}{\sqrt{1 - s_k}}\,.
$$
\begin{exercise}
For $\alpha_i>0$, $1\leq i\leq k$ show that
$$
\int\displaylimits_{\substack{0\leq x_1,\cdots,x_k\leq 1 \\ \sum x_i=1 }} d\vec{x} \prod_{i=1}^{k} x_i^{\alpha_i-1} = \frac{\prod_{i=1}^k \Gamma(\alpha_i)}{\Gamma(\sum_{i=1}^k \alpha_i)}.
$$
This is called the Dirichlet($\alpha$) distribution.
\end{exercise}
Using the exercise we conclude that
\[\E\big[I_k^2(t,x)\big] = \frac{t^{k/2}}{(4\pi)^{k/2}} \frac{\Gamma(1/2)^k}{\Gamma(k/2)}p^2(t,x).\qedhere\]
\end{proof}

Since $\Gamma(k/2) \sim (k/2)!$ this implies that the chaos series  converges in $L^2(\Omega,\mathcal{F},\mathbb{P})$ thus completing step 1.\medskip

\noindent{\bf Step (2):}
By Lemma \ref{lemIk} and the It\^{o} isometry,
$$
\E\big[z^2(t,x)\big] = \sum_{k=0}^{\infty} \E\big[I_k^2(t,x)\big] = p^2(t,x)\sum_{k=0}^{\infty} \frac{t^{k/2}}{(4\pi)^{k/2}} \frac{\Gamma(1/2)^k}{\Gamma(k/2)} \leq C p^2(t,x),
$$
where the constant $C=C(T)$ can be chosen fixed as $t$ varies in $[0,T]$.\medskip

\noindent{\bf Steps (3) and (4):}
Assume $z(t,x)$ solves SHE with delta initial data. We will show that $z(t,x)$ equals the chaos series. The mild form of SHE implies that
$$
z(t,x) = p(t,x) + \int_0^t \int_{\R} p(t-s,x-y) z(s,y)\xi(s,y)dyds.
$$
Iterate this (like Picard iteration) to obtain
\begin{align*}
z(t,x) = \sum_{k=0}^n I_k(t,x) + \int_{\Delta_{n+1}(t)}\int_{\R^{n+1}} &p(t-s_{n+1},x-y_{n+1})\cdots\\
&\quad\cdots p(s_2-s_1,y_2-y_1) z(s_1,y_1)\xi^{\otimes k}(\vec{s},\vec{y})d\vec{s}d\vec{y}.
\end{align*}
Comptue the $L^2(\Omega,\mathcal{F},\mathbb{P})$ norm of the remainder as
$$
\int_{\Delta_{n+1}(t)}\int_{\R^{n+1}} p^2(t-s_{n+1},x-y_{n+1})\cdots p^2(s_2-s_1,y_2-y_1)\cdot \E\big[z^2(s_1,y_1)\big]d\vec{y}d\vec{s}.
$$
Use bound $\E[z^2(s_1,y_1)]\leq cp^2(s_1,y_1)$ and notice that the above expression is then bounded as
$$
\leq c\,\E[I_{n+1}(t,x)] = c\frac{t^{(n+1)/2}}{(4\pi)^{n/2}} \frac{\Gamma(1/2)^{n+1}}{\Gamma(n/2)}p^2(t,x),
$$
which goes to zero as $n\rightarrow\infty$ hence showing that $z$ is the chaos series and that the chaos series solves the mild form of SHE.
\end{proof}

We close by demonstrating one use of the chaos series to probe the short time behavior of $z(t,x)$.
\begin{corollary}
The following one-point\footnote{Extending to space-time process convergence is not much harder.} weak convergence  holds:
$$
\lim_{\epsilon\rightarrow 0} \epsilon^{-1/4} [z(\epsilon t,\epsilon^{1/2}x) - p(\epsilon t,\epsilon^{1/2}x)] = \int_0^t \int_{\R} p(t-s,x-y)\xi(s,y)p(s,y)dyds.
$$
\end{corollary}
\begin{proof}
This comes from the first term in the chaos series whereas our estimate on $\E[I_k(t,x)^2]$ enables us to discard all subsequent terms. Recall
$$
I_1(t,x) = \int_0^t \int_{\R} p(t-s,x-y)\xi(s,y)p(s,y)dyds.
$$
Apply scalings to study now $\epsilon^{-1/4}I_1(\epsilon t, \epsilon^{1/2} x)$.

\begin{exercise}
Show that in distribution,
$$
\xi(t,x) = \epsilon^{(z+1)/2} \xi(\epsilon^z t, \epsilon x).
$$
Use it to show that $\ep^{-1/4}I_1(\ep t,\ep^{1/2} x) = I_1(t,x)$ in distribution  whereas we have $\ep^{-1/4} I_k(\ep t,\ep^{1/2} x) = \ep^{a_k} I_k(t,x)$. Compute $a_k>0$.
\end{exercise}
From this exercise, it is now easy to conclude the corollary.
\end{proof}

Unfortunately, the chaos series is not very useful in studying large $t$ behaviors (e.g. for $\log z(t,x)$) and also the chaos series do not behave very well under applications of functions. However, they can be useful in proving approximation as we will soon see.


\section{Scaling to the KPZ equation}\label{L3}
The heat equation with deterministic potential $v(t,x)$ is written in differential form as
$$
\partial_t z = \frac{1}{2}\partial_{xx} z + vz.
$$

\begin{exercise}
Show that if $h:=\log z$ and $u:=\partial_x h$ then
\begin{align*}
\partial_t h &= \frac{1}{2}\partial_{xx}h + \frac{1}{2}(\partial_x h)^2 + v, \\
\partial_t u &= \frac{1}{2}\partial_{xx}u + \frac{1}{2}\partial_x (u^2) + \partial_x v.
\end{align*}
\end{exercise}
The first equation is a Hamilton Jacobi equation with forcing and the second equation is a viscous Burgers equation with conservative forcing.

Define the mollifier $p_\kappa(x)=p(\kappa,x)$ (i.e. the heat kernel at time $\kappa$). Note that as $\kappa\rightarrow 0$, this approaches a delta function as its maximal height grows like $\kappa^{-1/2}$ and standard deviation shrinks like $\kappa^{1/2}$. We could just as well work with any similar mollifier sequence. Define mollified white-noise via
$$
\xi_\kappa(t,x) = \int_{\R} p_\kappa(x-y)\xi(t,y)dy = (p_\kappa * \xi)(t,x).
$$
By the It\^{o} isometry,
$$
\E\big[\xi_\kappa(t,x)\xi_\kappa(s,y)\big] = \delta_{t=s} c_k(x-y), \quad c_\kappa(x) = (p_\kappa * p_\kappa)(x).
$$
Note that as $\kappa\rightarrow 0$, we have $c_\kappa(x)\rightarrow\infty$.

\begin{exercise}
Define $z_\kappa$ to be the solution (same definition of solution as before) to
$$
\partial_t z_\kappa = \frac{1}{2}\partial_{xx} z_\kappa + \xi_\kappa z_\kappa,
$$
and  $h_\kappa := \log z_\kappa$. Show that
$$
\partial_t h_\kappa = \frac{1}{2}\partial_{xx} h_\kappa + \big[\frac{1}{2}(\partial _x h_\kappa)^2 - c_\kappa(0)\big] + \xi_\kappa.
$$
Also show that $z_\kappa\rightarrow z$ as $\kappa\searrow 0$. (Suggestion: use the chaos series.)
\end{exercise}
The exercise shows that in some sense, we can think of $h=\log z$ as the solution to
$$
\partial_t h = \frac{1}{2}\partial_{xx}h + \big[\frac{1}{2}(\partial _x h_k)^2 - \infty\big] + \xi.
$$
Since when studying interface fluctuations we generally need to subtract off overall height shifts, this $-\infty$ should not concern us so much. This suggests the definition
\begin{definition}
The Hopf--Cole solution $h:\R_{+}\times\R\rightarrow \R$ to the KPZ equation with $h_0:\R\rightarrow\R$ initial data is defined as $h=\log z$ where $z$ is the solution to the SHE with initial data $z_0(x)=e^{h_0(x)}$.
\end{definition}
For some initial conditions such as $\delta_{x=0}$, it does not make sense to take its logarithm. However, for all positive $t$ and $x\in \R$, we know that $z(t,x)$ is strictly positive (see \cite{Mueller,Gregorio}), hence the logarithm is well defined. The solution to KPZ corresponding to this delta initial data is called the narrow wedge solution (or KPZ with narrow wedge initial data). The idea is that in short time the log of $p(t,x)$ looks like $-x^2/2t$ which is a steep parabola. Hence the initial data corresponds to growth out of a wedge type environment. We will see a justification of this later.

Now let us ask: what kind of discrete models converge to the KPZ equation? To motivate this, let us do some rescaling of the KPZ equation.

Define the scaled KPZ solution (we temporarily use $z$ here as the ``dynamic scaling exponent'' in accordance with other literature)
$$
h^{\ep}(t,x) = \ep^{b} h(\ep^{-z}t,\ep^{-1}x).
$$
We find that under this scaling,
\begin{align*}
\partial_t h &= \ep^{z-b} \partial_t h^{\ep} \\
\partial_x h &= \ep^{1-b}\partial_x h^{\ep},\qquad (\partial_x h)^2 = \ep^{2-2b}(\partial_x h^{\ep})^2\\
\xi(t,x) &\stackrel{(d)}{=} \ep^{-(z+1)/2}\xi(\ep^{-z}t,\ep^{-1}x).
\end{align*}
Note that above, when writing $\partial_t h$ and $\partial_x h$ we mean to take the derivatives and then evaluate them at $\ep^{-z}t$ and $\ep^{-1}x$ respectively.
Thus
$$
\partial_t h^{\ep} = \frac{1}{2}\ep^{2-z} \partial_{xx} h^{\ep} + \frac{1}{2} \ep^{2-z-b} (\partial_x h^{\ep})^2 + \ep^{b-z/2+1/2}\xi.
$$
Note that the noise on the right-hand side is not the same for different $\ep$, but in terms of its distribution it is.

We now pose a natural question: are there any scalings of the KPZ equation under which it is invariant? If so, then we can hope to scale growth processes in the same way to arrive at the KPZ equation. However, one can check that there is no way to do this! On the other hand, there are certain weak scalings which fix the KPZ equation. By weak, we mean that simultaneously as we scaling time, space and fluctuations, we also can put tuning  parameters in front of certain terms in the KPZ equation and scale them with $\ep$. In other words, we simultaneously scale time, space and fluctuations, as well as the model. Let us consider two such weak scalings.\medskip

{\bf Weak non--linearity scaling:} Take $b=1/2,z=2$. The first and third terms stay fixed, but the middle term blows up. Thus, insert the constant $\lambda_{\ep}$ in front of the nonlinear term $(\partial xh)^2$ and set $\lambda_{\ep}=\ep^{1/2}$. Under this scaling, the KPZ equation is mapped to itself.\medskip

{\bf Weak noise scaling:} Take $b=0,z=2$. Under this scaling, the linear $\partial_{xx} h$ and nonlinear $(\partial_x h)^2$ terms stay fixed, but now the noise blows up. So insert $\beta_{\ep}$ in front of the noise term and set $\beta_{\ep}=\ep^{1/2}$, and again the KPZ equation stays invariant.\medskip

One can hope that these rescalings are attractive. In other words, if you take models with a parameter (nonlinearity or noise) that can be tuned, that these models will all converge to the same limiting object. We will see that weak non--linearity scaling can be applied to particle growth processes and weak noise scaling can be applied to directed polymers.

We would be remise at this point if we did not mention the broader notion of KPZ universality (see \cite{ICreview,QuastelReview} for surveys on this subject). It was predicted by \cite{FNS,KPZ} is that under $b=1/2$ and $z=3/2$ scaling (called KPZ scaling) the KPZ equation should have non-trivial limiting behavior. The exact nature of this limiting behavior (called the KPZ fixed point in \cite{CQfixedpt}) is still somewhat nebulous as we will briefly observe. The fixed point should be the limit of particle growth processes and directed polymers. To understand more about the KPZ fixed point, we will need to develop exact solvability, which will be done in later sections.\index{Kardar-Parisi-Zhang (KPZ)!universality}%

Here is a fact (which can be shown, for instance from the approximation of the KPZ equation from ASEP). If we start KPZ with two--sided Brownian motion, then later it will be a two--sided Brownian motion plus a height shift. One can see this by showing that particle systems preserve random walks. This suggests $b=1/2$ (the Brownian scaling). In this case, the only way to avoid blowup is to take $z=3/2$. So this suggests that KPZ scaling is, indeed, a natural candidate. But in this case, it looks like the noise and the linear term disappears. What would be left would be the inviscid Burgers equation. But this is not true, since, for example, the inviscid Burgers equation is deterministic and also since the KPZ equation must preserve Brownian motion, but the inviscid Burgers equation does not.

\begin{figure}[ht]
\begin{center}
\includegraphics[scale=.8]{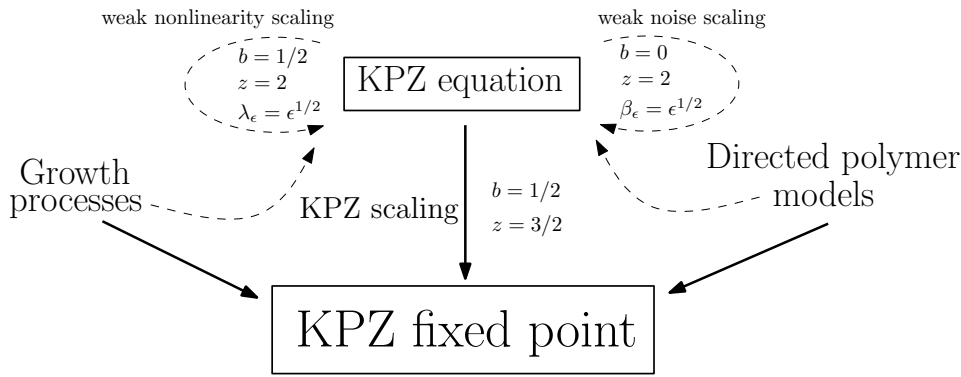}
\end{center}
\caption{Three types of scalings for the KPZ equation. Weak noise and weak non--linearity scaling fix the equation whereas under KPZ scaling, the KPZ equation should go to the (nebulous) KPZ fixed point. It is believed (and in some cases shown) that these ideas extend to a variety of growth processes and directed polymer models.}
\end{figure}

Because we presently lack a stochastic description of the KPZ fixed point, most of what we know about it comes from asymptotic analysis of a few exactly solvable models. Moreover, most information comes in the form of knowledge of certain marginal distributions of the space-time KPZ fixed point process.

In order to illustrate the ideas of weak universality of the KPZ equation, lets consider the corner growth model, which is equivalent to ASEP.
Growth of local minima into local maxima occurs after independent exponentially distributed waiting times of rate $p$ and the opposite (local maxima shrinking into local minima) occurs with rate $q$, with $p+q=1$ and $p-q=: \gamma >0$. The step initial data corresponds to setting $h^{\ga}(0,x) = \vert x\vert$. Here we use $h^{\ga}(t,x)$ to record the height function at time $t$ above position $x$, for the model with asymmetry $\ga$.

\begin{figure}[ht]
\begin{center}
\includegraphics[scale=0.8]{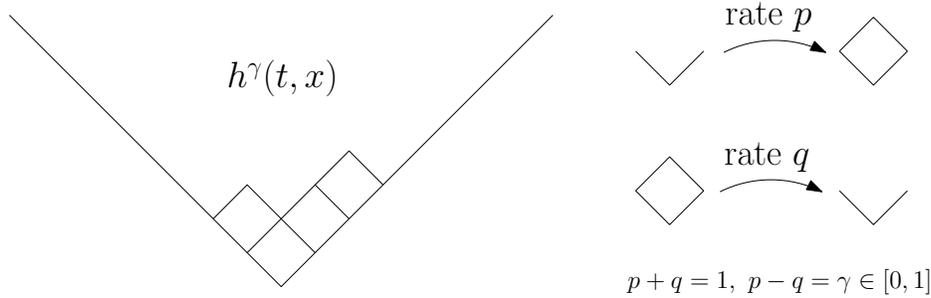}
\end{center}
\caption{Corner growth model, or equivalently asymmetric simple exclusion process (ASEP) started from step initial data.}
\end{figure}

The following theorem is an example of how the KPZ equation arises under weak non--linearity scaling.

\begin{theorem}[\cite{ACQ}] For the step initial data corner growth model,
$$
z_{\ep}(t,x) = \frac{1}{2}\ep^{-1/2}\exp\left( -\ep^{1/2} h^{\ep^{1/2}}(\ep^{-2}t,\ep^{-1}x)  + \ep^{-3/2}\frac{t}{2}\right)
$$
converges to the solution to the SHE with $\delta_{x=0}$ initial distribution. Convergence here means as probability measures on $C([\tilde{\delta},T],\C(\R))$ (continuous functions of time to the space of continuous spatial trajectories) for any $\tilde{\delta}, T>0$.
\end{theorem}

What about weak noise scaling? One example of how this leads to the KPZ equation is through directed polymers, a class of models introduced in \cite{HuHe}. To motivate this, lets return to the heat equation with deterministic potential:
$$
\partial_t z = \frac{1}{2}\partial_{xx} z + vz.
$$
Recall the Feynman-Kac representation. \index{Feynman-Kac representation}%
Consider a Brownian motion run backwards in time, fixed so that $b(t)=x$, and let $\mathcal{E}_{b(t)=x}$ represent the expectation with respect to this Brownian motion. Then
$$
z(t,x) = \mathcal{E}_{b(t)=x}\left[ z_0(b(0)) \exp\left( \int_0^t v(s,b(s))ds \right) \right].
$$
\begin{exercise}
Prove this. Hint: This can be proved via a chaos series in $v$.
\end{exercise}

\begin{figure}[ht]
\begin{center}
\includegraphics[scale=.6]{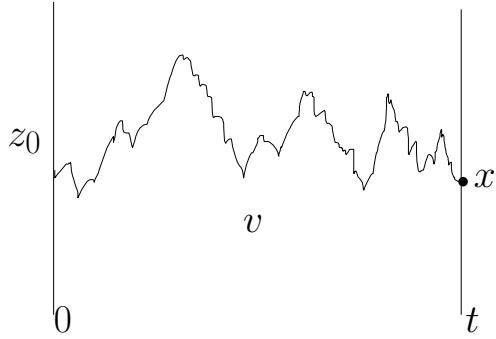}
\end{center}
\caption{Path integral solution to heat equation with potential $v$.}
\end{figure}

Even if $z_0$ is not a function, this can still make sense. For example, if $z_0(x) = \delta_{x=0}$ then the entire expression is killed if you don't end up at $0$. Thus we can replaced Brownian motion with a Brownian bridge at the cost of a factor of $p(t,x)$.

What if $v=\xi_\kappa=(p_\kappa * \xi)(t,x)$? Then the solution, $z_\kappa$, to the mollified equation
$$
\partial_t z_\kappa = \frac{1}{2}\partial_{xx}z_\kappa + \xi_\kappa z_\kappa
$$
can be represented via the following path integral
$$
z_\kappa(t,x) = \mathcal{E}_{b(t)=x} \left[ z_0(b(0)) \exp\left(\int_0^t \xi_\kappa(s,b(s))ds - \frac{c_\kappa(0)t}{2} \right)\right],
$$
where $c_\kappa(x) = (p_\kappa * p_\kappa)(x)$. See \cite{BC} for more on this.

Let us discretize this whole picture. We will discretize the white noise by i.i.d. random variables and the Brownian motion by a simple symmetric random walk. Consider up/right paths $\pi$ on $(\Z_+)^2$. Let $w_{ij}$ be i.i.d. random variables and $\beta>0$ be inverse temperature. Let the partition function be
$$
Z^{\beta}(M,N) = \sum_{\pi:(1,1)\rightarrow (M,N)} e^{\beta\sum_{(i,j)\in \pi} w_{ij}},
$$
where the summation is over all paths $\pi$ (as in Figure \ref{fig4}) which start at $(1,1)$, and take unit steps in the $(1,0)$ or $(0,1)$ directions and terminate at $(M,N)$.

This satisfies a discrete SHE
$$
Z^{\beta}(M,N) = e^{\beta w_{M,N}} \big( Z^{\beta}(M-1,N) + Z^{\beta}(M,N-1)\big).
$$
The similarity to the SHE becomes more apparent if you make the change of variables $Z^{\beta}(M,N) = 2^{M+N}\tilde{Z}^{\beta}(M,N)$, $e^{\beta w_{i,j}}=1+\xi_{i,j}$ and iterate the recursion once more. The resulting equation can then be organized as
\begin{align*}
&\tilde{Z}^{\beta}(M,N)- \tilde{Z}^{\beta}(M-1,N-1) \\
&= \frac{1}{2}\Big(\tilde{Z}^{\beta}(M-2,N)-2\tilde{Z}^{\beta}(M-1,N-1)+\tilde{Z}^{\beta}(M,N-2)\Big)
\!+\!\textrm{multiplicative noise},
\end{align*}
where $\textrm{multiplicative noise}$ is the sum of products of $\xi$'s multiplied by single $\tilde{Z}^{\beta}$ terms. If $M+N$ is interpreted as time, and $M-N$ as space, this looks like a discrete SHE.

\begin{figure}[ht]
\begin{center}
\includegraphics[scale=.6]{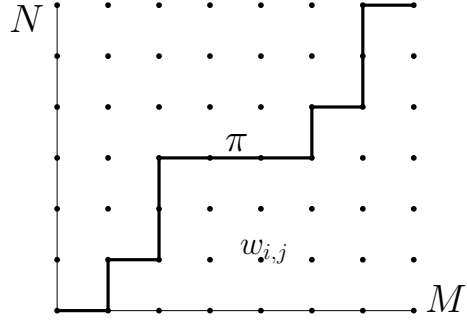}
\end{center}
\caption{Discrete polymer path $\pi$ moving through random potential of $w_{i,j}$.}\label{fig4}
\end{figure}

\begin{theorem} [\cite{AKQ}] Let $w_{i,j}$ have finite eighth moment. Then there exist explicit constants $C(t,x,N)$ (depending on the distribution of $w_{i,j}$) such that
$$
Z_N(t,x) := C(t,x,N)Z^{2^{-1/2}N^{-1/4}}(tN+x\sqrt{N}, tN - x\sqrt{N})
$$
converges as $N\rightarrow \infty$ to $z(2t,2x)$ where $z$ is the solution of the SHE with $\delta_{x=0}$ initial data.
\end{theorem}
Here $N$ is playing the role of $\ep^{-2}$ from before, and this is readily recognized as a weak noise scaling limit.

In the next section we will consider a semi-discrete polymer in which $M$ has been taken large and the discrete sums of $w_{i,j}$ become Brownian increments. For a modification of this semi-discrete polymer model we will give a proof of an analogous weak noise scaling limit.

\section{Weak noise convergence of polymers to SHE}\label{L4}
The O'Connell--Yor partition function (also called the semi--discrete polymer or semi-discrete SHE) 
\index{Semi-discrete stochastic heat equation (SHE)}%
\index{Stochastic heat equation (SHE)!semi-discrete}%
at inverse temperature $\beta>0$ is defined by
$$
Z^{\beta}(T,N) := \int\displaylimits_{0\leq s_1<s_2<\cdots<s_{N-1}<T} d\vec{s} e^{ \beta\big(B^N_1(0,s_1) + B^N_2(s_1,s_2) + \ldots B^N_N(s_{N-1},T)\big)},
$$
where $B^N_i(a,b)=B^N_i(b)-B^N_i(a)$ with $\{B^N_i\}_{i=1}^{N}$ independent Brownian motions.

\begin{figure}[ht]
\begin{center}
\includegraphics[scale=.6]{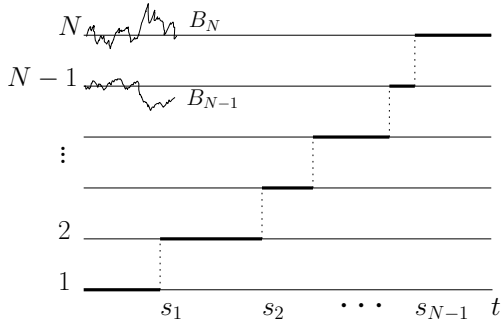}
\end{center}
\caption{O'Connell--Yor polymer partition function.}
\end{figure}

\begin{theorem}[\cite{QRMF} (see also \cite{Mihai}]\label{Thmoy}
There exist explicit constants $C(t,x,N)$ such that
$$
Z_N(t,x) := C(t,x,N)Z^{N^{-1/4}}(tN+x\sqrt{N}, tN)
$$
converges as $N\rightarrow \infty$ to $z(t,x)/p(t,x)$ where $z$ is the solution of the SHE with $\delta_{x=0}$ initial data and $p(t,x)$ is the heat kernel. Moreover, there exists a coupling of the Brownian motions $B^N_j$ with $\xi$ so the convergence occurs in $L^2(\Omega,\mathcal{F},\mathbb{P})$ as well.
\end{theorem}

This theorem is proved by developing $Z^{\beta}(T,N)$ into a semi-discrete chaos series. \index{Chaos series}%
Let us, therefore, recall the continuous chaos series for the solution to the SHE with delta initial data:
$$
\frac{z(1,0)}{p(1,0)} = \sum_{k=0}^{\infty} \tilde{I}_k(1,0)
$$
where, due to the division by $p(1,0)$ we have slightly modified chaos terms
$$
\tilde{I}_k(1,0) = \int_{\Delta_k} d\vec{s} \int_{\R^k} d\vec{x} \rho_k(\vec{s},\vec{x}) \xi^{\otimes k}(\vec{s},\vec{x}).
$$
Here we have again used the notation
$$
\Delta_k(t) = \{0\leq s_1 < \cdots < s_k < t\}, \qquad \textrm{with} \quad \Delta_k(1) =: \Delta_k
$$
and
$$
\rho_k(\vec{s},\vec{x}) = \frac{p_{k;1,0}(\vec{s},\vec{x})}{p(1,0)} = \frac{p(1-s_k,x_k)p(s_k-s_{k-1},x_k-x_{k-1})\cdots p(s_1,x_1)}{p(1,0)}.
$$
Note that $\rho_k(\vec{s},\vec{x})$ represents the transition density of a Brownian bridge from $(0,0)$ to $(1,0)$ going through the points $\vec{x}$ at times $\vec{s}$. We will not prove Theorem \ref{Thmoy}, but instead will prove a result about a slightly modified polymer partition function, which is itself a key step in proving Theorem \ref{Thmoy}.

\begin{definition}
The modified semi-discrete polymer partition function is defined as
$$
\tilde{Z}^{\beta}(T,N) = \left| \Delta_{N-1}(T) \right|^{-1} \int_{\Delta_{N-1}(T)}  d\vec{s} \prod_{j=1}^N \big(1 + \beta B_j(s_{j-1},s_j)\big),
$$
where by convention $s_0=0$ and $s_N=T$.
\end{definition}

Note that for a measurable set $A\in \R^n$, we use $|A|$ represent the Lebesgue measure of the set. Also, observe that for small $\beta$ we have the first order expansion $e^{\beta B_j(s_{j-1},s_j)} \approx 1 + \beta B_j(s_{j-1},s_j)$.

\begin{theorem}[\cite{QRMF}]\label{Thmodified}
As $N\rightarrow \infty$, we have the following weak convergence
$$
\tilde{Z}^{N^{-1/4}}(N,N) \rightarrow \frac{z(1,0)}{p(1,0)}.
$$
Moreover, there exists a coupling of the $B_j^N$ with $\xi$ so that this convergence occurs in $L^2(\Omega,\mathcal{F},\mathbb{P})$.
\end{theorem}
Note that we have included superscript $N$ on the $B^N_j$ since, though for each $N$ $\big\{B^N_j\big\}_{j=1}^{N}$ are marginally independent Brownian motions, as $N$ varies these collections of curves are not independent and, in fact, are built from the same probability space on which $\xi$ is defined. The $L^2(\Omega,\mathcal{F},\mathbb{P})$ convergence certainly implies the weak convergence.

\begin{proof}[Proof of Theorem \ref{Thmodified}]
We proceed in three steps. In Step (1) we describe the coupling of the $B^N_j$ with $\xi$. In Step (2) we express $\tilde{Z}^{\beta}(N,T)$ as a semi--discrete chaos series and then using the coupling to express this as a continuous chaos series in $\xi$. Finally, in Step (3) we use convergence of the Poisson jump process to Brownian bridge to conclude termwise convergence of the chaos series, and the convergence of the entire chaos series too.

\noindent{\bf Step (1):} Define the function
$$
\varphi^N(s,x) = \left( \frac{s}{N}, \frac{x-s}{\sqrt{N}} \right).
$$
This maps the square $[0,N]^2$ to the rhombus $\{ (t,y): 0\leq t\leq 1, -t\sqrt{N} \leq y \leq (1-t)\sqrt{N} \}$. Call
$$
I_j^N(t) = \left[ \frac{j-1}{\sqrt{N}} - \sqrt{N}t, \frac{j}{\sqrt{N}} - \sqrt{N}t  \right].
$$

\begin{figure}[ht]
\begin{center}
\includegraphics[scale=.69]{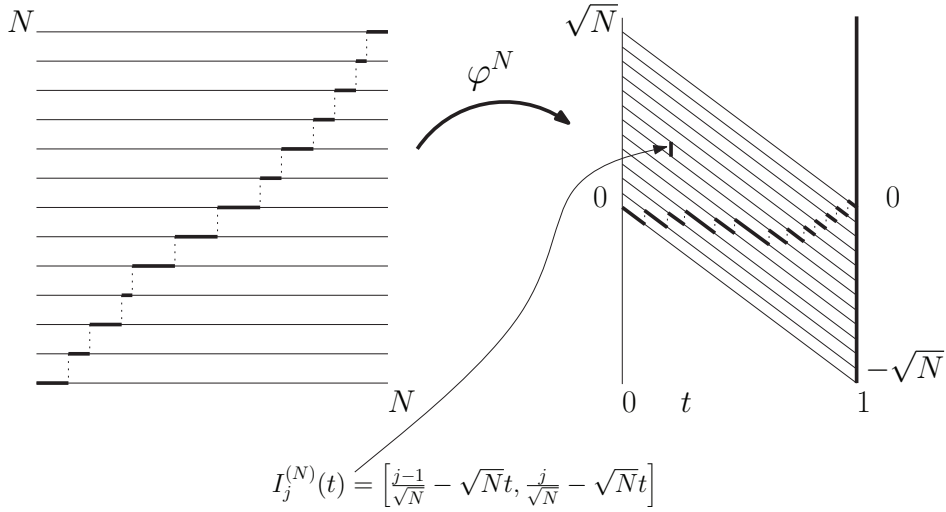}
\end{center}
\caption{The mapping  $\varphi^N$ used to build the coupling of the $B^N_j$ with $\xi$.}
\end{figure}

\begin{exercise}
Show that since $\varphi^N$ has Jacobian $N^{-3/2}$, if we define
$$
B_j^N(s_1,s_2) = N^{3/4} \int_{s_1/N}^{s_2/N} ds \int_{I_j^N(s)} dx \xi(s,x),
$$
then
$ \{B_j^N(0,s)\}_{j\in [N], s\in [0,N]}$ are independent standard Brownian motions.
\end{exercise}
We will use this coupling.

\noindent{\bf Step (2):} Let us rewrite $\tilde{Z}^{\beta}(T,N)$ via an expectation.  Let $X_{\bullet}$ be distributed as the trajectory on $[0,T]$ of a Poisson jump process with $X_0=1$ (i.e. starting at 1 and increasing by 1 after independent exponential waiting times).

\begin{exercise}
Show that conditioning $X_{\bullet}$ on the event that $X_T=N$, the jump times of $X_{\bullet}$ (labeled $s_1,\ldots,s_{N-1}$) are uniform over the simplex $0<s_1<\ldots<s_{N-1}<T$ with measure
$$
\frac{d\vec{s}}{\left| \Delta_{N-1}(T) \right|} \mathbf{1}_{\vec{x}\in \Delta_{N-1}(T)}.
$$
\end{exercise}

We may rewrite
\begin{align*}
\tilde{Z}^{\beta}(T,N) &= \E\left[ \prod_{j=1}^N \Big(1 + \beta \int_0^T \mathbf{1}_{X_s=j} dB_j(s) \Big)\, \Big\vert X_T=N \right],\\
&= \sum_{k=0}^N \beta^k \tilde{J}_k(T,N) \\
\tilde{J}_k(T,N) &= \E\left[ \int_{[0,T]^k} \sum_{\vec{i} \in D_k^N} \prod_{j=1}^k \mathbf{1}_{X_{s_j}=i_j} dB
_{i_j}(s_j) \vert X_T=N \right],
\end{align*}
where
$$
D_k^N = \{i \in \Z_+^k: 1\leq i_1 < i_2 < \cdots < i_k \leq N\}.
$$
and (for use later)
$$
\partial D_k^N = \{i \in \Z_+^k: 1 \leq i_1 \leq i_2 \leq \cdots \leq i_k \leq N \text{ with at least one equality}\}.
$$
The expectation can be taken inside yielding
$$
\tilde{J}_k(T,N) = \int_{[0,T]^k} \sum_{\vec{i} \in D_k^N} \mathbb{P}(X_{s_1}=i_1,\ldots,X_{s_k}=i_k \vert X_T=N) dB_{i_1}(s_1)\cdots dB_{i_k}(s_k).
$$
We would like to use the coupling now, so it makes sense to change our probabilities into the variables of the image of $\varphi^N$. Call
$$
\rho^N(s,t,x,y) = \mathbb{P}\left(X_{N(t-s)} = \lceil \sqrt{N}y + Nt \rceil - \lceil \sqrt{N}x + Ns \rceil \right)
$$
and
$$
\rho_k^N(\vec{s},\vec{x}) = \frac{N^{k/2}\rho^N(s_k,1;x_k,0)\prod_{j=1}^k \rho^N(s_{j-1},s_j,x_{j-1},x_j)}{\rho^N(0,1;0,0)},
$$
where $\vec{s} \in \Delta_k$ and $\vec{x}\in \R^k$. This is the semi--discrete analog of $\rho_k(\vec{s},\vec{x})$. The scaling $N^{k/2}$ comes from local limit theorems. For $f\in L^2(\Delta_k\times \R^k)$ define
$$
\Vert f\Vert_2^2 = \int_{\Delta_k} \int_{\R^k} f(\vec{s},\vec{x})^2 d\vec{s}d\vec{x}.
$$

Using our coupling and the formula for $\rho_k^N$ we can rewrite
$$
\tilde{J}_k(N,N) = N^{k/4} \int_{[0,1]^k} d\vec{s} \sum_{\vec{i} \in D_k^N} \prod_{j=1}^k \int_{I_{i_j}^N(s_j)} dx_j \rho_k^N(\vec{s},\vec{x}) \xi^{\otimes k}(\vec{s},\vec{x}).
$$
The $N^{1/4}$ is the product of $N^{-1/2}$ from the definition of $\rho_k^N$ and $N^{3/4}$ from the coupling.

\noindent{\bf Step (3):} Let us start by stating two lemma (we will not prove them here, though their proofs are not so hard).
\begin{lemma}\label{lemma1}
There exists a constant $c\in (0,\infty)$ such that for all $k,N\geq 1$,
$
\Vert \rho_k^N \Vert^2_2 \leq c^k \Vert \rho_k\Vert_2^2.
$
\end{lemma}
\begin{lemma}\label{lemma2}
There exists a constant  $c\in (0,\infty)$  such that for all $N\geq 1$ we have
$
\sum_{k=1}^N \Vert \rho_k^N\Vert_2^2 \leq C
$
and for each $k\geq 1$,
$\lim_{N\rightarrow\infty} \Vert \rho_k^N - \rho_k \Vert_2^2 = 0.$
\end{lemma}

Given these lemmas, we deduce the following (which we prove).
\begin{lemma}\label{lemconv} The following convergence occurs in $L^2(\Omega,\mathcal{F},\mathbb{P})$:
\[
\lim_{N\rightarrow\infty} N^{-k/4} \tilde{J}_k(N,N) = \tilde{I}_k(1,0).
\]
\end{lemma}
\begin{proof}
Observe that
\begin{align*}
N^{-k/4}\tilde{J}_k(N,N) &= \int_{[0,1]^k} \sum_{\vec{i}\in [N]^k} \prod_{j=1}^k \int_{I_{i_j}^N(s_j)} dx_j \rho_k^N(\vec{s},\vec{x})\xi^{\otimes k}(\vec{s},\vec{x}) \\
&- \int_{[0,1]^k}\sum_{\vec{i} \in \partial D_k^N} \prod_{j=1}^k \int_{I_{i_j}^N(s_j)} dx_j \rho_k^N(\vec{s},\vec{x})\xi^{\otimes k}(\vec{s},\vec{x})
\end{align*}
In the first term we can replace $[0,1]^k$ by $\Delta_k(1)$ and the $\sum\prod\int$ can be replaced by $\int_{\R^k}$ (all since $\rho_k^N$ vanishes otherwise). By Lemma \ref{lemma2}, we can conclude that the first term limits to $\tilde{I}_k(1,0)$. Call $E_k^N$ the second term. By Lemma \ref{lemma1}
\begin{equation}\label{ENK}
\E\big[(E_k^N)^2\big] \leq c^k \int_{\Delta_k} d\vec{s} \sum_{i\in \partial D_k^N} \prod_{j=1}^k \int_{I_{i_j}^N(s_j)} dx_j (\rho_k(\vec{s},\vec{x}))^2 = \mathcal{O}(c^k/N).
\end{equation}
The $1/N$ factor can be understood as the cost of staying in the same level for two samples of times.
\end{proof}

We can now complete the proof of Theorem \ref{Thmodified}.
From Lemma \ref{lemconv} we know that for each $M\geq 1$,
$$
\lim_{N\rightarrow \infty} \sum_{k=0}^{M} N^{-k/4}\tilde{J}_k(N,N) = \sum_{k=0}^{M} \tilde{I}_k(1,0)
$$
with convergence in $L^2(\Omega,\mathcal{F},\mathbb{P})$. From Section \ref{L3} we know that for all $\ep>0$ there exists $M\geq 1$ such that
$$
\sum_{k=M}^{\infty} \tilde{I}_k(1,0) <\ep.
$$

Thus, it remains to show that for all $\ep>0$ there exists $M\geq 1$ such that
$$
\sum_{k=M}^{\infty} N^{-k/4}\tilde{J}_k(N,N) <\ep,
$$
once $N>M$. By Lemma \ref{lemma1} we can bound each term on the left-hand side of the above expression by a constant times
$$
\Vert \rho^N_k(\vec{t},\vec{x})\Vert_2^2 + \E\big[(E_k^N)^2\big].
$$
By Lemma \ref{lemma1} and equation (\ref{ENK}) these are both bounded by constants times $\Vert \rho_k(\vec{t},\vec{x})\Vert_2^2$. This shows that
$$
\sum_{k=M}^{\infty} N^{-k/4}\tilde{J}_k(N,N) \leq c \sum_{k=M}^{\infty} \tilde{I}_k(1,0)
$$
and hence taking $M$ large enough this can be bounded by $\ep$ as desired.
\end{proof}

\section{Moments of q--TASEP via duality and Bethe ansatz}\label{L5} \index{Bethe ansatz}%
The past few sections have focused on making sense of the KPZ equation (via the SHE) and then deriving various approximations for it. We turn to one such approximation which enjoys certain exact solvability. This exact solvability goes through in the limit to the KPZ equation / SHE. In that context, the exact solvability amounts to the ability to write down simple and exact formulas for the expectation (over the white noise $\xi$) of the solution to the SHE with delta initial data. Unfortunately, the moment problem for the SHE is not well-posed, meaning the knowledge of these moments does not determine its distribution. Thus, despite the many formulas, this does not provide a direct rigorous route to exactly solving the equation. By going to the below discretization, we preserve this exact solvability while also resolving the issue of the moment problem well-posedness.

The interacting particle system $q$--TASEP will be the main subject of the remaining sections (for more on this model, see \cite{BorCor,BCS,ICreview,ICreviewICM}). It is a model for traffic on a one--lane road where cars slow down as they approach the next one. The scale on which this occurs is controlled by $q\in (0,1)$.

Particles are labeled right to left as $x_1(t)>x_2(t)>\ldots,$ where $x_i(t)$ records the location of particle $i$ at time $t$. In continuous time each particle can jump right by one according to independent exponential clocks of rate $1-q^{\text{gap}},$ where gap is the number of empty sites before the next right particle. In other words, for particle $i$ the gap is $x_{i-1}-x_i-1.$

\begin{figure}[ht]
\begin{center}
\includegraphics[scale=.5]{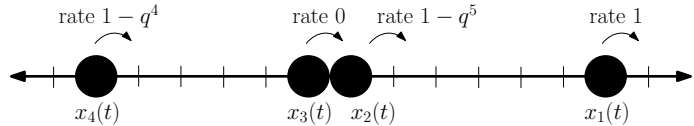}
\end{center}
\caption{The $q$--TASEP with particles labeled $x_i(t)$ and jump rates written in.}
\end{figure}

Since the jumping of particle $i$ only depends on $i-1$, it suffices to consider an $N$--particle restriction. In this case, the state space is given by
$$
X^N :=\big\{\vec{x}=(x_0,\ldots,x_N): \infty=x_0>x_1>\ldots>x_N\big\}.
$$
The role of fixing $x_0=\infty$ is to have a virtual particle to make notation nicer. This system can be encoded as a continuous time Markov process on $X^N$.

Recall the following facts about Markov processes on $X$  (see \cite{Lig} for more background on Markov processes and interacting particle systems).
A Markov process is defined via a semigroup $\{S_t\}_{t\geq 0}$ with $S_{t_1}S_{t_2}=S_{t_1+t_2}$ and $S_0=\mathrm{Id},$ acting on a suitable domain of functions $f:X\rightarrow \R$. For $f:X\rightarrow \mathbb{R}$ in the domain of $S,$ define $\mathbb{E}^x\big[f(x(t))\big]=S_tf(x),$ where $X$ is the state space and $\mathbb{E}^x$ is expectation with respect to starting state $x(0)=x$. The generator of a Markov process is defined as
$$
L=\lim_{t\rightarrow 0} \frac{S_t-\mathrm{Id}}{t}
$$
and captures the process since
$$
S_t = e^{tL} = \sum_{k\geq 0} \frac{1}{k!}(tL)^k.
$$
It follows then that $\tfrac{d}{dt} S_t = S_t L = LS_t,$ so for $f:X\rightarrow \mathbb{R}$
$$
\frac{d}{dt} \mathbb{E}^x\big[f(x(t))\big] = \mathbb{E}^x\big[Lf(x(t))\big] = L\mathbb{E}^x\big[f(x(t))\big].
$$

The generator of $q$--TASEP acts on $f:X^N\rightarrow\mathbb{R}$ as
$$
(L^{q-TASEP}f)(\vec{x}) := \sum_{i=1}^N (1-q^{x_{i-1}-x_i-1})\left(f(\vec{x_i}^+) - f(\vec{x})\right),
$$
where $\vec{x_i}^+=(x_0,x_1,\ldots,x_{i}+1,\ldots,x_N)$ represents the movement of particle $i$ by 1 to the right. Note that we will not worry here about specifying the domain of the $q$--TASEP generator.

Consider the functions $f_1,\ldots,f_N$ defined by $f_n(\vec{x})=q^{x_n}$.  How can $q^{x_n(t)}$ change in an instant of time?
\begin{align*}
dq^{x_n}(t) &= (1-q^{x_{n-1}-x_n-1})(q^{x_n+1}-q^{x_n})dt + \text{ noise  (i.e. martingale) }\\
&=(1-q)(q^{x_{n-1}-1}-q^{x_n})dt + \text{ noise }.
\end{align*}
Now taking expectations gives
$$
\frac{d}{dt} \mathbb{E}[q^{x_n(t)}] = (1-q)\left( \mathbb{E}[q^{x_{n-1}(t)-1}] - \mathbb{E}[q^{x_n(t)}] \right).
$$
which shows that as a function of $n\in\{1,\ldots,N\}$ the expectations $\mathbb{E}[q^{x_n(t)}]$ solve a triangular, closed system of linear ordinary differential equations.

In order to generalize this system, and then ultimately solve it, we will use a Markov duality.

\begin{definition}
Suppose $x(\cdot),y(\cdot)$ are independent Markov processes with state spaces $X,Y$ and generators $L^X,L^Y$. Let $H:X\times Y\rightarrow \mathbb{R}$ be bounded and measurable. Then we say $x(\cdot)$ and $y(\cdot)$ are dual with respect to $H$ if for all $x\in X,y\in Y$, $L^XH(x,y)=L^YH(x,y)$.
\end{definition}
\begin{exercise}
Show that duality implies $\E^x[H(x(t),y)]=\E^y[H(x,y(t))]$ for all $t$ and hence
$$
\frac{d}{dt} \E^x\big[H(x(t),y)\big]= L^X\E^x\big[H(x(t),y)\big] = L^Y \E^x \big[ H(x(t),y)\big].
$$
\end{exercise}

It turns out that $q$--TASEP enjoys a simple duality with a Markov process called the $q$--Boson process \cite{SasWad,BCS} (which is a special totally asymmetric zero range process) with state space
$$
Y^N =\big\{ \vec{y} = (y_0,\ldots,y_N) \vert y_i\in \mathbb{Z}_{\geq 0}\big\},
$$
in which there can be $y_i$ particles above site $i$ and in continuous time. According to independent exponential clocks, one particle moves from $i$ to $i-1$ at rate $1-q^{y_i}$, nothing enters site $N$ and nothing exits site $0$ (so the total number of particles is conserved over time).

\begin{figure}[ht]
\begin{center}
\includegraphics[scale=.5]{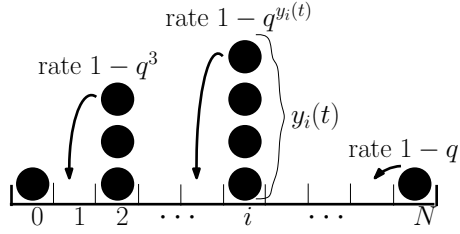}
\end{center}
\caption{The $q$--Boson process.}
\end{figure}

The generator of the $q$--Boson process acts on functions $h:Y^N\rightarrow\mathbb{R}$ as
$$
(L^{\text{q--Boson}}h)(\vec{y}) := \sum_{i=1}^N (1-q^{y_i})\big(h(\vec{y}^{i,i-1}) -  h( \vec{y})\big),
$$
where $\vec{y}^{i,i-1}=(y_0,y_1,\ldots,y_{i-1}+1,y_i-1,\ldots,y_N)$ represents moving a particle from $i$ to $i-1$. Note that $q$--TASEP gaps evolve according to the $q$--Boson jump rates. So, in a sense, the duality we now present is actually a self-duality.

\begin{theorem}[\cite{BCS}]
As Markov processes, $q$--TASEP $\vec{x}(t)\in X^N$ and $q$--Boson process $\vec{y}(t)\in Y^N$ are dual with respect to
$$
H(\vec{x},\vec{y}) = \prod_{i=0}^N q^{(x_i+i)y_i}
$$
with the convention that $H=0$ if $y_0>0$ and otherwise the product starts at $i=1$.
\end{theorem}
\begin{proof}
For all $\vec{x},\vec{y}$,
\begin{align*}
(L^{\text{q--TASEP}}H)(\vec{x},\vec{y}) & = \sum_{i=1}^N (1-q^{x_{i-1}-x_{i}-1})(H(\vec{x}_i^+,\vec{y}) - H(\vec{x},\vec{y}))\\
&= \sum_{i=1}^N (1-q^{x_{i-1}-x_{i}-1})(q^{y_i}-1)\prod_{j=0}^N q^{(x_j+j)y_j}\\
&= \sum_{i=1}^N (1-q^{y_i})(H(\vec{x},\vec{y}^{i,i-1})-H(\vec{x},\vec{y}))\\
&= (L^{\text{q--Boson}}H)(\vec{x},\vec{y})
\end{align*}
\end{proof}

We will use this theorem to solve for joint moments of $q^{x_n(t)+n}$ as $n$ varies in $\{1,\ldots,N\}$. In particular if we treat $\vec{x}$ as fixed then the duality implies
\begin{equation}\label{Star}
\frac{d}{dt} \E^{\vec{x}}\left[\prod_{i=0}^N q^{(x_i(t)+i)y_i}\right] = L^{\text{q--Boson}}\E^{\vec{x}}\left[\prod_{i=0}^N q^{(x_i(t)+i)y_i}\right],
\end{equation}
where $L^{\text{q--Boson}}$ acts in the $\vec{y}$ variables.

\begin{proposition}\label{partA}
Fix $q$--TASEP initial data $\vec{x}\in X^N$. If $h: \mathbb{R}_{\geq 0}\times Y^N\rightarrow\mathbb{R}$ solves

(1) For all $\vec{y}\in Y^N$ and $t\in \mathbb{R}_{\geq 0}$
$$
\frac{d}{dt} h(t;\vec{y}) = L^{\text{q--Boson}}h(t;\vec{y});
$$

(2) For all $\vec{y}\in Y^N$,
$$
h(0;\vec{y}) = h_0(\vec{y}) := H(\vec{x},\vec{y});
$$
Then for all $\vec{y}\in Y^N$ and $t\in \mathbb{R}_{\geq 0}$, $\E^{\vec{x}}\big[H(\vec{x}(t),\vec{y})\big]= h(t;\vec{y})$.
\end{proposition}
\begin{proof}
By \eqref{Star}, Proposition \ref{partA}(1) must hold for $\E^{\vec{x}}[H(\vec{x}(t),\vec{y})]$ and Proposition \ref{partA}(2) follows by definition. Uniqueness follows because $L^{\text{q--Boson}}$ preserves the number of particles and restricts to a triangular system of coupled ODEs on each $k$--particle subspace. Then, standard ODE uniqueness results \cite{Cod} imply uniqueness. In other words, the system is closed due to particle conservation of the $q$--Boson process.
\end{proof}

Since the $q$--Boson process acts on $k$--particle subspaces, it preserves the state spaces
$$
Y_k^N = \big\{\vec{y} \in Y^N: \sum y_i = k\big\}.
$$
For $\vec{y}\in Y_k^N$ we can associate a vector of weakly ordered particle locations in
$$W_{\geq 0}^k = \big\{\vec{n}=(n_1\geq\ldots\geq n_k\geq 0)\big\}.$$
Write $\vec{n}(\vec{y})$ and $\vec{y}(\vec{n})$ for this association. For example, if $N=3$ and $k=4$ and $\vec{y}=(y_0=0,y_1=3,y_2=0,y_3=1)$ then $\vec{n}=(3,1,1,1)$. We will abuse notation and write $h(t,\vec{n}):=h(t, \vec{y}(n)).$

In order to solve the system of equations for $h$, its more convenient to work in these $n$ coordinates. Let us consider how the system looks for different $k$.

For $k=1$, $\vec{n}=(n)$ and the evolution equation becomes
$$\tfrac{d}{dt} h(t;\vec{n}) = (1-q) \nabla h(t;\vec{n}),$$
where $(\nabla f)(n) = f(n-1) - f(n)$. This is just the generator of a single $q$--Boson particle moving.

For $k=2$, $\vec{n}=(n_1\geq n_2)$ and we must consider two cases
\begin{itemize}
\item If $n_1>n_2,$ then we have for Proposition \ref{partA}(1)
$$
\frac{d}{dt} h(t;n_1,n_2) = \sum_{i=1}^2 (1-q)\nabla_i h(t;n_1,n_2).
$$
\item
If $n_1=n_2=n$ then
$$
\frac{d}{dt} h(t;n,n) = (1-q^2)\nabla_2h(t;n,n),
$$
where we have chosen $\nabla_2$ in order to preserve the order of $\vec{n}$ so as to stay in $W_{\geq 0}^{k}$. Unfortunately this is not constant coefficient or separable so it is not a priori clear how to solve it. Moreover, as $k$ grows the number of boundary cases may grow quite rapidly like $2^{k-1}$.
\end{itemize}

To resolve this, we use an idea of Bethe \cite{Bethe} from 1931. We try to rewrite in terms of solution to $k$ particle free evolution equation subject to $k-1$ two--body boundary conditions. Usually this is not possible and one has many body boundary conditions, but if it is possible then we say the system is coordinate Bethe ansatz solvable. Let us see this idea in motion. \index{Bethe ansatz}%

For $k=2$ consider $u: \mathbb{R}_{\geq 0}\times \mathbb{Z}_{\geq 0}^2 \rightarrow \mathbb{R}$ which satisfies
$$
\frac{d}{dt} u(t;\vec{n}) = \sum_{i=1}^2 (1-q)\nabla_i u(t;\vec{n}).
$$
Then, when $n_1>n_2$ this right--hand--side exactly matches that of the true evolution equation. However, for $n_1=n_2=n$ the two right hand sides differ by
\begin{equation}\label{Star2}
\sum_{i=1}^2 (1-q)\nabla_i u(t;n,n) - (1-q^2)\nabla_2 u(t;n,n).
\end{equation}
If we could find a $u$ such that $\eqref{Star2}\equiv 0$ then when we restrict $u$ to $\{\vec{n}:n_1\geq n_2\},$ it will actually solve the true evolution equation: $u\vert_{\vec{n}\in W_{\geq 0}^2}=h$. The boundary condition is equivalent to
$$
(\nabla_1-q\nabla_2)u\vert_{n_1=n_2} \equiv 0.
$$
The only way we might hope to find such a $u$ is to tinker with the initial data outside the set $W_{\geq 0}^2$. In a sense, this is like an advanced version of the reflection principle.

For $k=3$, we might need more than two body boundary conditions (e.g. $n_1=n_2=n_3$), but amazingly all higher order cases follow from the two body cases. Hence we have a coordinate Bethe ansatz solvable system. This is shown by the following result.

\begin{proposition}\label{partB}
If $u:\mathbb{R}_{\geq 0}\times \mathbb{Z}_{\geq 0}^k \rightarrow \mathbb{R}$ solves

(1) For all $\vec{n}\in \mathbb{Z}_{\geq 0}^k$ and $t\in \mathbb{R}_{\geq 0}$,
$$
\frac{d}{dt} u(t;\vec{n}) = \sum_{i=1}^k (1-q)\nabla_i u(t;\vec{n});
$$

(2) For all $\vec{n} \in \mathbb{Z}_{\geq 0}^k$ such that $n_i=n_{i+1}$ for some $i\in \{1,\ldots,k-1\}$ and all $t\in \mathbb{R}_{\geq 0}$,
$$
(\nabla_i - q\nabla_{i+1})u(t;\vec{n})=0;
$$

(3) For all $\vec{n} \in W_{\geq 0}^k, u(0;\vec{n})=h_0(\vec{n});$

Then for all $\vec{n}\in W_{\geq 0}^k, h(t;\vec{n})=u(t;\vec{n})$.
\end{proposition}
\begin{exercise}
Prove the proposition. Hint: if $n_1=\cdots =n_c>n_{c+1}$ then use the boundary condition to replace $\sum_{i=1}^{c} (1-q)\nabla_i$ by $(1-q^c)\nabla_c$, and use this to rewrite the free evolution equation into $L^{\text{q--Boson}}$ (when written in terms of the $\vec{n}$ variables).
\end{exercise}

It is possible to solve this for general $h$, cf. \cite{BCPS13}. We will focus only on the particular case $h_0(\vec{n})= \prod_{i=1}^k \mathbf{1}_{n_i>0}$. This initial data corresponds (via Proposition \ref{partA}) with studying step initial data $q$--TASEP where $x_i(0)=-i$ for $i=1,\ldots, N$. Then $H\big(\vec{x},\vec{y}(\vec{n})\big) = \prod_{i=1}^k q^{x_{n_i}(t)+n_i}=1$ as long as all $n_i\geq 0$. Thus, taking $h_0(\vec{n})=\prod_{i=1}^k \mathbf{1}_{n_i>0}$ and solving the system in Proposition \ref{partB} we find that for $\vec{n}\in W_{\geq 0}^{k}$,
$$
\mathbb{E}^{\text{step}}\left[ \prod_{i=1}^k q^{x_{n_i}(t)+n_i}\right]=u(t;\vec{n}).
$$
The idea for solving this (in fact for general initial data) also traces back to Bethe \cite{Bethe}.
First, solve the one particle free evolution equation (fundamental solution). Then, use linearity to take superpositions of fundamental solutions in order to try to satisfy boundary and initial conditions. It is not a priori obvious that this will work ... but it does.

\section{Exact formulas and the replica method}\label{L6}
\index{Replica method}%
The goal of the first part of this section is to explicitly solve the true evolution equation with initial data $h_0(\vec{n})=\prod_{i=1}^k \mathbf{1}_{n_i>0}.$ This will yield a formula for $\E^{\text{step}}[\prod_{i=1}^k q^{x_{n_i}(t)+n_i}].$ We will utilize Proposition \ref{partB} which reduces this to solving the free evolution equation subject to two--body boundary conditions. Let us start with the case $k=1$ (for which there are no boundary conditions). For $z\in \mathbb{C}\backslash\{1\}$ define
$$
u_z(t;n):= \frac{e^{(q-1)tz}}{(1-z)^n}.
$$
It is immediate to check that Proposition \ref{partB}(1) is satisfied:
$$
\frac{d}{dt} u_z(t;n) = (1-q)\nabla u_z(t;n).
$$
Since $k=1$, Proposition \ref{partB}(2) is not present so it remains to check 3). Note that linear combinations of $u_z(t;n)$ over $z$'s still solve Proposition \ref{partB}(1). Consider then
$$
u(t;n) = \frac{-1}{2\pi i} \oint u_z(t;n) \frac{dz}{z},
$$
where the contour contains $1$ but not $0$. When $t=0$ and $n\geq 1$ expand contours to infinity. We have at least $z^{-2}$ decay at infinity hence no residue (recall from \cite{Alhfors}) at infinity. However, we crossed a first order pole at $z=0$ which gives $u(0;n)\equiv 1$. When $n\leq 0$, there is, on the other hand, no pole at at $z=1$, hence the integral is necessarily zero. This shows Proposition \ref{partB}(3).

Hence we have proved that
$$
\E^{\text{step}}\big[ q^{x_{n}(t)+n}\big] = u(t;n) = \frac{-1}{2\pi i} \oint\frac{e^{(q-1)tz}}{(1-z)^n}  \frac{dz}{z}.
$$
We will now state the general $k$ version of this result. This may appear as something of a rabbit pulled out of a hat. The origins of this formula are demystified in Section \ref{secspec}.

\begin{theorem}\label{thmgenmom}
Proposition \ref{partB} with $h_0(\vec{n}) = \prod_{i=1}^k \mathbf{1}_{n_i>0}$ is solved by
$$
u(t;\vec{n}) = \frac{(-1)^k q^{k(k-1)/2}}{(2\pi i)^k} \oint\cdots\oint \prod_{1\leq A<B\leq k} \frac{z_A-z_B}{z_A-qz_B} \prod_{j=1}^k u_{z_j}(t;n_j)\frac{dz_j}{z_j},
$$
with contours such that the $z_A$ contour contains $\{qz_B\}_{B>A}$ and $1$ but not $0.$
\end{theorem}

\begin{figure}[ht]
\begin{center}
\includegraphics[scale=1.5]{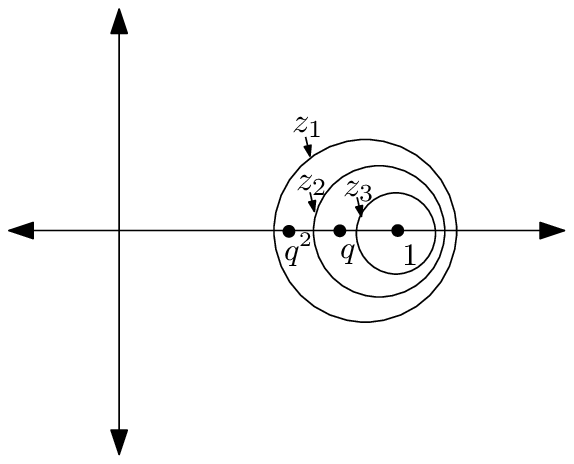}
\end{center}
\caption{Nested contours for Theorem \ref{thmgenmom}. \index{Nested contour integral}%
}
\end{figure}

\begin{corollary}
For all $\vec{n}\in W_{\geq 0}^k$,
\begin{align*}
&\E^{\text{step}}\left[\prod_{i=1}^k q^{x_{n_i}(t)+n_i}\right]
= u(t;\vec{n})\\
&\qquad= \frac{(-1)^k q^{k(k-1)/2}}{(2\pi i)^k} \oint\cdots\oint \prod_{1\leq A<B\leq k}\frac{z_A-z_B}{z_A-qz_B} \prod_{j=1}^k \frac{e^{(q-1)tz_j}}{(1-z_j)^{n_j}} \frac{dz_j}{z_j}.
\end{align*}
\end{corollary}
Since $q\in (0,1)$ and $x_n(t)+n\geq 0,$ the observable $q^{x_n(t)+n}\in (0,1].$ Hence, knowledge of the joint moments of $\{q^{x_n(t)+n}\}_{n\in \{1,\ldots,N\}}$ for fixed $t$ completely characterizes their distribution, and hence also the distribution of $\{x_n(t)\}_{n\in \{1,\ldots,N\}}!$ Later, we will see how to extract some useful distributional formulas from these moments.

\begin{exercise}
Prove Theorem \ref{thmgenmom}. Parts (1) and (3) clearly generalize the $k=1$ case. For (2) study the effect of applying $\nabla_i-q\nabla_{i+1}$ to integrand, and show integral of result is $0$.
\end{exercise}
In the remaining part of this section I will study how $q$--TASEP and its duality and moment formulas behave as $q\nearrow 1.$ Doing so, we will encounter a semi--discrete stochastic heat equation 
\index{Semi-discrete stochastic heat equation (SHE)}%
\index{Stochastic heat equation (SHE)!semi-discrete}%
or equivalently the O'Connell--Yor semi--discrete directed polymer model. In our next section we will return to $q$--TASEP and study some distributional properties.

From duality we know that the $q$--TASEP dynamics satisfy (for some Martingale $M_n(t)$)
\begin{align*}
dq^{x_n(t)+n} &= (1-q)\nabla q^{x_n(t)+n}dt + q^{x_n(t)+n}dM_n(t)\\
q^{x_n(0)+n}&\equiv \mathbf{1}_{n\geq 1} \text{ (for step initial data)}.
\end{align*}
This suggests that as $q\nearrow 1$ the (properly scaled) observable $q^{x_n(t)+n}$ might converge to the semi--discrete SHE.

\begin{definition}
A function $z:\mathbb{R}_{\geq 0}\times \mathbb{Z}_{\geq 0}\rightarrow \R_{\geq 0}$ solves the semi--discrete SHE with initial data $z_0:\Z_{\geq 0}\rightarrow \R_{\geq 0}$ if:
\begin{align*}
dz(\tau;n) &= \nabla z(\tau;n)d\tau + z(\tau;n)dB_n(\tau) \\
z(0;n) &= z_0(n).
\end{align*}
\end{definition}
\begin{theorem}[\cite{BorCor}]\label{thmlimits}
For $q$--TASEP with step initial data, set
$$
q=e^{-\epsilon}, \qquad t=\epsilon^{-2}\tau, \qquad x_n(t) = \epsilon^{-2}\tau - (n-1)\epsilon^{-1}\log\epsilon^{-1}-\epsilon^{-1}\mathcal{F}_{\epsilon}(\tau;n)
$$
(this third equation above defines $\mathcal{F}_{\epsilon}$) and call $z_{\epsilon}(\tau;n) = \exp\big\{-3\tau/2 + \mathcal{F}_{\epsilon}(\tau;n)\big\}.$ Then as a space time process $z_{\epsilon}(\tau;n)$ converges weakly to $z(\tau;n)$ with $z_0(n)=\mathbf{1}_{n=1}$ initial data.
\end{theorem}

\begin{proof}[Heuristic proof sketch]
The initial data is clear since $z_{\epsilon}(0;n) = \epsilon^{n-1}e^{\epsilon n} \rightarrow \mathbf{1}_{n=1}.$

As for the dynamics, observe that for $d\tau$ small,
\begin{align*}
d\mathcal{F}_{\epsilon}(\tau;n) &\approx \mathcal{F}_{\epsilon}(\tau;n) - \mathcal{F}_{\epsilon}(\tau-d\tau;n)\\
&= \epsilon^{-1}d\tau - \epsilon\big(x_n(\epsilon^{-2}\tau) - x_n(\epsilon^{-2}\tau - \epsilon^{-2}d\tau)\big)
\end{align*}
Under scaling, the $q$--TASEP jump rates are given by
$$
1-q^{x_{n-1}(t)-x_n(t)-1} = 1 -\ep e^{\mathcal{F}_{\ep}(\tau;n-1)-\mathcal{F}_{\ep}(\tau;n)} + O(\ep^2).
$$
So in time $\ep^{-2}d\tau,$ by convergence of Poisson point process to Brownian motion
\begin{align*}
&\ep\big(x_n(\ep^{-2}\tau) - x_n(\ep^{-2}\tau - \ep^{-2}d\tau)\big) \\
&\qquad\qquad\approx \ep^{-1}d\tau - e^{\mathcal{F}_{\ep}(\tau;n-1) - \mathcal{F}_{\ep}(\tau;n)}d\tau - (B_n(\tau)-B_n(\tau-d\tau)).
\end{align*}
Thus we see that (using little oh notation)
$$
d\mathcal{F}_{\ep}(\tau;n) \approx e^{\mathcal{F}_{\ep}(\tau;n-1) - \mathcal{F}_{\ep}(\tau;n)}d\tau + dB_n(\tau) + o(1)
$$
Exponentiating and applying It\^{o}'s lemma gives
$$
de^{\mathcal{F}_{\ep}(\tau;n) } = \left( \frac{1}{2}e^{\mathcal{F}_{\ep}(\tau;n) } + e^{\mathcal{F}_{\ep}(\tau;n-1) }\right)d\tau + e^{\mathcal{F}_{\ep}(\tau;n) }dB_n(\tau) + o(1)
$$
or going to $z_{\ep},$
$$
dz_{\ep}(\tau;n) = \nabla z_{\ep}(\tau;n)d\tau + z_{\ep}(\tau;n)dB_n(\tau)  + o(1)
$$
and as $\ep\searrow 0$ we recover the semi--discrete SHE.
\end{proof}

Theorem \ref{thmlimits} implies that for $q=e^{-\ep},$
$$
q^{x_n(\ep^{-2}\tau)}e^{\ep^{-1}\tau}\ep^{n-1}e^{-3\tau/2} \longrightarrow z(\tau;n).
$$
Though not shown in the proof, it is quite reasonable to imagine (and perhaps prove, without too much additional work) moments converge as well. We will check this (somewhat indirectly) in two steps. First, we will take limits of our known $q$--TASEP moments. Second, we will compute directly the semi--discrete SHE moments. Consider
\begin{multline}\label{Star3}
\E^{\text{step}}\left[ \prod_{i=1}^k q^{x_n(\ep^{-2}\tau)} e^{\ep^{-1}\tau}\ep^{n_i-1} e^{-3\tau/2}\right] \\
= \frac{(-1)^k q^{k(k-1)/2}}{(2\pi i)^k} \oint\cdots\oint \prod_{1\leq A<B\leq k} \frac{z_A-z_B}{z_A-qz_B} \prod_{j=1}^k \frac{e^{(q-1)\ep^{-2}\tau z_j + \ep^{-1}\tau-3\tau/2}}{(\ep^{-1}(1-z_j))^{n_j}} \frac{\ep^{-1}dz_j}{dz_j}.
\end{multline}
with  $q=e^{-\ep}$.  Change variables $z=e^{-\ep \tilde{z}},$ then as $\ep\searrow 0$
\begin{align*}
\frac{z_A-z_B}{z_A-qz_B} &\rightarrow \frac{\tilde{z}_A - \tilde{z}_B}{\tilde{z}_A-\tilde{z}_B-1},\\
\frac{1}{(\ep^{-1}(1-z_j))^{n_j}} &\rightarrow \frac{1}{(\tilde{z}_j)^{n_j}},\\
\frac{\ep^{-1}dz_j}{z_j} &\rightarrow -d\tilde{z}_j ,\\
(q-1)\ep^{-2}\tau z_j + \ep^{-1}\tau - 3\frac{\tau}{2}  &\rightarrow \tau (\tilde{z}_j-1)
\end{align*}
and hence
$$
\lim_{\ep\searrow 0} \eqref{Star3} = \frac{1}{(2\pi i)^k} \oint\cdots\oint \prod_{1\leq A<B<\leq k} \frac{z_A-z_B}{z_A-z_B-1} \prod_{j=1}^k \frac{e^{\tau (z_j-1)}}{z_j^{n_j}}dz_j,
$$
with the $z_A$ contour containing $\{z_B+1\}_{B>A}$ and $0.$

\begin{figure}[ht]
\begin{center}
\includegraphics[scale=1]{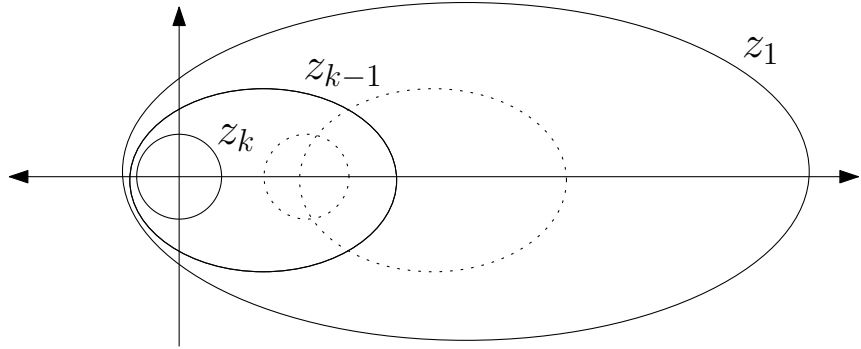}
\end{center}
\caption{Nested contours \index{Nested contour integral}%
for semi-discrete SHE moment formulas. The dashed contours represent the image under addition by 1 of the $z_k$ and $z_{k-1}$ contours. Notice that the $z_{k-1}$ contour includes both the $z_k$ contour and its image under addition by 1. This is the nature of the nesting.}
\end{figure}

We will now check that
$$
\lim_{\ep\searrow 0} \eqref{Star3}= \E\left[\prod_{j=1} z(\tau;n_j)\right].
$$
This is done through a limit version of the duality for $q$--TASEP. This now goes by the name of the polymer replica method. \index{Replica method}%
For this we should recall the Feymann--Kac representation, as briefly do now.

Consider a homogeneous Markov process generator $L$ and deterministic potential $v$. We will provide a path integral or directed polymer interpretation for the solution to
\begin{align*}
\frac{d}{dt} z(t,x) &= (Lz)(t,x) + v(t,x)z(t,x) \\
z(0,x) &= z_0(x).
\end{align*}
Let $\varphi(\cdot)$ be the Markov process with generator $L$ run backwards in time from $t$ to $0$ and let $\mathcal{E}^{t,x}$ be the associated expectation operator. For example, for $L=\nabla,$ trajectories of $\varphi$ look like in Figure \ref{figFKone}.

\begin{figure}[ht]
\begin{center}
\includegraphics[scale=.7]{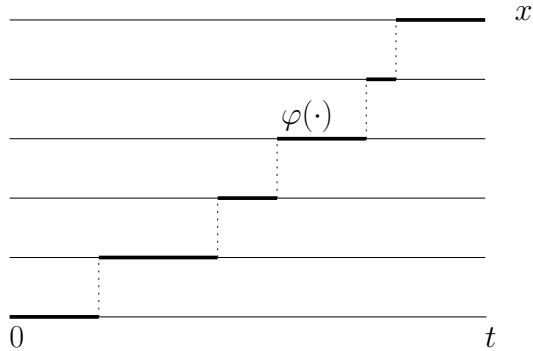}
\end{center}
\caption{A possible trajectory in the Feynman-Kac representation.}\label{figFKone}
\end{figure}

Let $p(t,x)=\mathcal{E}^{t,x}\left[\mathbf{1}_{\varphi(0)=0}\right]$ be the heat kernel for $L.$

For $v\equiv 0,$ by superposition / linearity of expectation we have
$$
z(t,x) = \mathcal{E}^{t,x}\left[z_0(\varphi(0))\right].
$$
When $v$ is turned on, Duhamel's principle allows us to write
$$
z(t,x) = \int_{\R} p(t,x-y)z_0(y)d\mu(y) + \int_0^t ds\int_{\R} d\mu(y) p(t-s,y-x) z(s,y)v(s,y),
$$
where $\mu$ is measure on the state space (e.g. Lebesgue or counting measure). This identity can be applied repeatedly to yield an infinite series (which is convergent under mild hypothesis) and which can be identified with
$$
z(t,x) = \mathcal{E}^{t,x}\left[z_0(\varphi(0))\cdot \exp\left\{ \int_0^t v(s,\varphi(s))d\mu(s)\right\}\right],
$$
which is called the Feynman-Kac representation. \index{Feynman-Kac representation}%

\begin{exercise}
Prove the Feynman-Kac representation using chaos series in $v$.
\end{exercise}

When $v$ is random (Gaussian) care is needed, as we must deal with stochastic integrals. This leads to something called the Wick or Girsanov correction. For $z(\tau;n)$ with $v(\tau;n)d\tau=dB_n(\tau)$ and $L=\nabla$ this yields
$$
z(\tau;n) = \mathcal{E}^{t,x}\left[z_0(\varphi(0))\exp\left\{ \int_0^{\tau} dB_{\varphi(s)}(s) - \frac{ds}{2} \right\}\right],
$$
which is called the O'Connell Yor polymer partition function. The study of general directed polymer models is quite interesting but will be too far off--topic.

\begin{figure}[ht]
\begin{center}
\includegraphics[scale=.7]{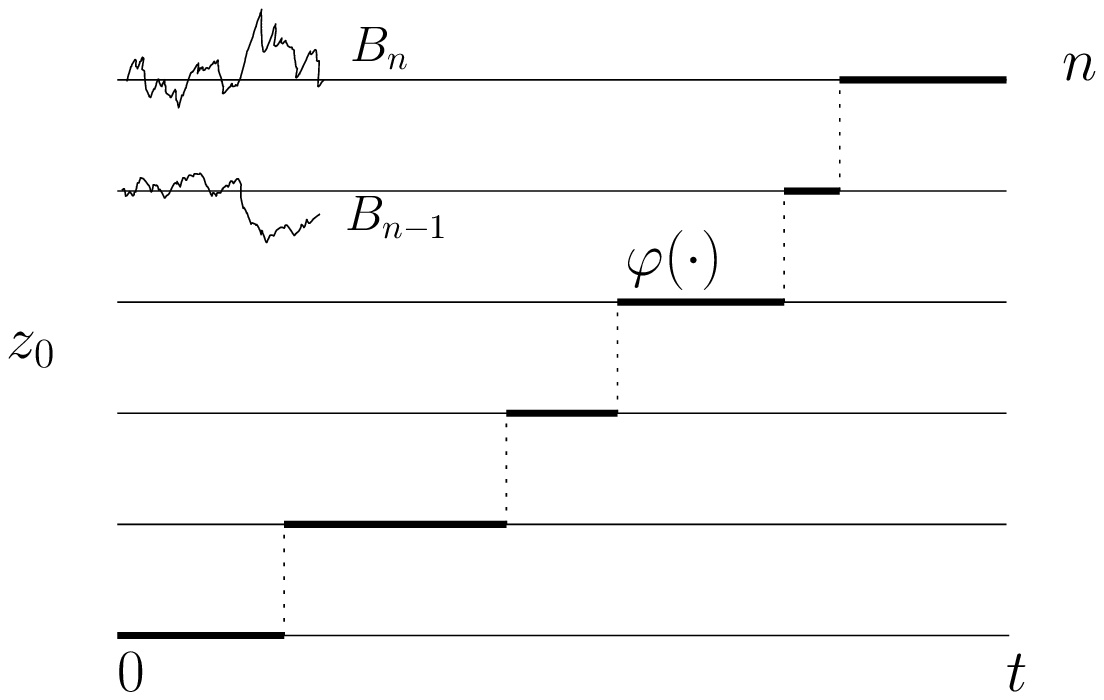}
\end{center}
\caption{The O'Connell Yor polymer partition function.}
\end{figure}

Define $\bar{z}(\tau;\vec{n})=\E\left[\prod_{i=1}^k z(\tau;n_i)\right]$ then applying the Feynman-Kac representation to each $z(\tau;n_i)$ with $\varphi_i(\cdot)$ and $\mathcal{E}_i$ associated to $z_i,$ we find
$$
\bar{z}(\tau;\vec{n})= \E\left[\prod_{i=1}^k \mathcal{E}_i^{\tau,n_i}\left[ z_0(\varphi_i(0))\cdot \exp\left\{\int_0^{\tau} dB_{\varphi_i(s)}(s)-\frac{ds}{2}\right\}\right]\right].
$$
By interchanging the $\E$ and $\mathcal{E}_1^{\tau,n_1}\cdots\mathcal{E}_k^{\tau,n_k}$ we get
$$
\bar{z}(\tau;\vec{n})= \mathcal{E}_1^{\tau,n_1}\cdots\mathcal{E}_k^{\tau,n_k} \left[\prod_{i=1}^k  z_0(\varphi_i(0)) \cdot \E\left[ \exp\left\{\int_0^t \sum_{i=1}^k dB_{\varphi_i(s)}(s)-\frac{ds}{2}\right\}\right]\right].
$$
Using $\E[e^{kX}]=e^{k^2\sigma^2/2}$ for $X\sim \mathcal{N}(0,\sigma^2)$ we find:
\begin{exercise}
$$
\E\left[ \exp\left\{\int_0^t \sum_{i=1}^k dB_{\varphi_i(s)}(s)-\frac{ds}{2}\right\}\right] = \exp\left\{\int_0^t \sum_{1\leq i<j\leq k} \mathbf{1}_{\varphi_i(s)=\varphi_j(s)} ds\right\}
$$
or in other words, the exponential of the pair local overlap time.
\end{exercise}
Thus
$$
\bar{z}(\tau;\vec{n}) = \mathcal{E}^{\tau;\vec{n}}\left[ \prod_{i=1}^k z_0(\varphi_i(0)) \exp\left\{ \int_0^{\tau} \sum_{1\leq i<j\leq k} \mathbf{1}_{\varphi_i(s)=\varphi_j(s)}ds \right\}\right].
$$
\begin{exercise}
Applying Feynman-Kac again shows that $z(\tau;\vec{n})$ with $\tau\in\R_{\geq 0},\vec{n}\in \Z_{>0}^k$ is the unique solution to the semi--discrete delta Bose gas:
\begin{align*}
\frac{d}{dt}\bar{z}(\tau;\vec{n}) &= \left(\sum_{i=1}^k \nabla_i + \sum_{1\leq i<j\leq k} \mathbf{1}_{n_i=n_j} \right) \bar{z}(\tau;\vec{n}),\\
\bar{z}(0;\vec{n}) &= \prod_{i=1}^k z_0(n_i).
\end{align*}
\end{exercise}
Notice that the above system involves functions that are symmetric in $n_1,\ldots,n_k$ whereas the nested contour integral formula we computed above for $\bar{z}(\tau;\vec{n})$ is only valid on $n_1\geq n_2\geq\ldots \geq n_k,$ and clearly not symmetric. It is possible to check that extending the integral formula symmetrically we get a solution to the above system. Alternatively we can show how the above delta Bose gas can be rewritten in terms of a free evolution equation with boundary conditions.

\begin{proposition}[\cite{BCS}]
If $u:\R_{\geq 0}\times \Z_{\geq 0}^k\rightarrow \R_{\geq 0}$ solves

(1) For all $\vec{n}\in (\Z_{\geq 0})^k$ and $\tau\in \R_{\geq 0}$,
$$
\frac{d}{dt} u(\tau;\vec{n}) = \sum_{i=1}^k \nabla_i u(\tau;\vec{n}).
$$

(2) For all $\vec{n}\in \Z_{\geq 0}^k$ such that for some $i \in \{1,\ldots,k-1\}, n_i=n_{i+1},$
$$
(\nabla_i-\nabla_{i+1}-1)u(\tau;\vec{n})=0.
$$

(3) For all $\vec{n}\in W_{\geq 0}^k$, $u(0;\vec{n})=\bar{z}_0(\vec{n}).$

Then for all $\vec{n}\in W_{\geq 0}^k,$ $u(\tau;\vec{n})=\bar{z}(\tau;\vec{n})$.
\end{proposition}

\begin{exercise}
Prove this proposition.
\end{exercise}

\begin{exercise}
Check our limiting integral formulas satisfy proposition, hence proving that their restriction to $\vec{n}\in W_{\geq 0}^k$ gives $\bar{z}(\tau;\vec{n}).$
\end{exercise}

\begin{exercise}
Reverse engineer the proof of the proposition in order to produce a delta Bose gas at the $q$--Boson level.
\end{exercise}

\section{Unnesting the contours and moment asymptotics}\label{L7}
In the last two sections we proved moment formulas for step initial data $q$--TASEP and $\mathbf{1}_{n=1}$ initial data semi--discrete SHE: 
\index{Semi-discrete stochastic heat equation (SHE)}%
\index{Stochastic heat equation (SHE)!semi-discrete}%
For $n_1\geq \ldots \geq n_k>0$
\begin{align*}
\mathbb{E}\left[ \prod_{j=1}^k q^{x_{n_j}(t)+n_j}\right] &= \frac{(-1)^k q^{k(k-1)/2}}{(2\pi i)^k} \oint\cdots\oint \prod_{1\leq A<B\leq k} \frac{z_A-z_B}{z_A-qz_B} \prod_{j=1}^k \frac{e^{(q-1)tz_j}}{(1-z_j)^{n_j}}\frac{dz_j}{z_j},\\
\mathbb{E}\left[ \prod_{j=1}^k z(\tau;n_j)\right] &= \frac{1}{(2\pi i)^k} \oint\cdots\oint \prod_{1\leq A<B\leq k} \frac{z_A-z_B}{z_A-z_B-1} \prod_{j=1}^k \frac{e^{\tau(z_j-1)}}{z_j^{n_j}}dz_j
\end{align*}
with nested contours.

We will now look to study the growth, for instance, of $\E\big[z(\tau,n)^k\big]$ as $\tau$ and $n$ go to infinity. This, it turns out, will require us to figure out how to unnest the contours so as to perform asymptotic analysis. But first, let us motivate this investigation with a story and a new concept -- intermittency.

Parabolic Anderson Models \index{Parabolic Anderson Model}%
are a class of models for mass transport or population dynamics in a disordered environment. We will be focusing on one spatial dimension, but as motivation consider algae in the sea. Algae can grow (or duplicate), die or move. The growth / death is very much a function of the local conditions (heat, sun light, nutrients, etc.) while the movement can be thought as due to steady currents.

We would like to understand certain intermittent behaviors, like the occurrence of large spikes in population for a little while in certain isolated areas.

Let us model the ocean as $\mathbb{Z}$ and think of algae as unit mass particles at the sites of $\mathbb{Z}$. There are no restrictions on particles per site, and in fact each particle evolves independent of others, only coupled through interaction with a common environment. In particular in continuous time, particles do three things independently:
\begin{itemize}
\item
split into two unit mass particles at exponential rate $r_+(\tau,n)$;
\item
die at exponential rate $r_-(\tau,n)$;
\item
jump to the right by $1$ at rate $1$.
\end{itemize}
The functions $r_+$ and $r_-$ represent an environment.

\begin{figure}[ht]
\begin{center}
\includegraphics[scale=1]{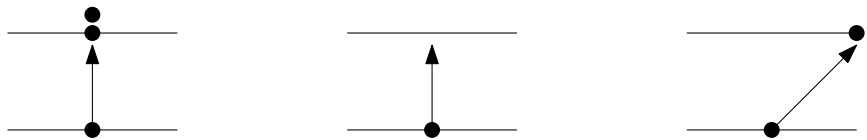}
\end{center}
\caption{Particles choose amongst: (Left) duplication; (center) death; (right) jump right.}
\end{figure}

Call $m(\tau,n)$ the expected total mass (expectation over the random evolution). Then one sees that
$$
\frac{d}{d\tau} m(\tau,n) = \nabla m(\tau,n) + \big(r_+(\tau,n)-r_-(\tau,n)\big)m(\tau,n).
$$
If the individual masses are very small, yet very many then this actually describes the mass density evolution.

If the environment $r_+,r_-$ is quickly and randomly evolving, then it may be appropriate to model $r_+(\tau,n)-r_-(\tau,n)$ by white noise in which case
$$
dm(\tau,n) = \nabla m(\tau,n)d\tau + dB_n(\tau)m(\tau,n)
$$
and we recover the semi--discrete SHE.

Initial data $m(0,n)=\mathbf{1}_{n=1}$ corresponds with starting a single cluster of particles at $n=1$ and letting them evolve.

It makes sense when studying this in large time, to scale $\tau$ and $n$ proportionally like $n=\nu\tau$ ($\nu=1$ making the most sense). We expect to see the overall population die out (though never reach $0$) at a certain exponential rate. However, we would like to understand how often or likely it is to see spikes above this background level. One measure of this is through Lyapunov exponents.

\begin{definition}
Assuming they exist, the almost sure Lyapunov exponent is
$$
\tilde{\gamma}_1(\nu) := \lim_{\tau\to\infty} \frac{1}{\tau}\log z(\tau,\nu\tau)
$$
and for $k\geq 1$, the $k$-th moment Lyapunov exponent is
$$
\gamma_k(\nu) := \lim_{\tau\to\infty} \frac{1}{\tau} \log \E\big[z(\tau,\nu \tau)^k\big].
$$
If these exponents are strictly ordered like
$$
\tilde{\gamma}_1(\nu) < \gamma_1(\nu) < \frac{\gamma_2(\nu)}{2} < \ldots < \frac{\gamma_k(\nu)}{k} < \ldots
$$
then the system is called intermittent.
\end{definition}

Intermittency implies that moments are only determined by tail behavior, and higher moments probe higher into the tails.

\begin{exercise}
Use Jensen's inequality to prove that there is always weak inequality among the Lyapunov exponents.
\end{exercise}
\begin{exercise}
Prove that if $z$ is intermittent then for any $\alpha$ such that $\tfrac{\gamma_k(\nu)}{k} < \alpha < \tfrac{\gamma_{k+1}(\nu)}{k+1},$
$$
\mathbb{P}\big(z(\tau,\nu \tau)>e^{\alpha\tau}\big) \leq e^{-(\alpha-\frac{\gamma_k(\nu)}{k})\tau}.
$$
\end{exercise}
Unfortunately, intermittency of $z$ will preclude deducing the distribution of $z$ from its moments. We will return later to a way around this.

We will explicitly compute all of these exponents, starting with the $\gamma_k$'s. For $\tilde{\gamma}_1,$ this will come later in a slightly different way.

In order to perform asymptotics of $\E\big[z(\tau, \nu\tau)^k\big],$ we should deform all contours to ones on which we can use steepest descent analysis \index{Steepest descent}%
(see Appendix \ref{asyana} for a simple worked out example). Such deformations may cross poles and hence pick up residues. We will see how this works.

Consider the case $k=1$:
$$
\E[z(\tau,\nu\tau)] = \frac{1}{2\pi i} \oint_{\vert z\vert =1} \frac{e^{\tau(z-1)}}{z^{\nu\tau}}dz = \frac{1}{2\pi i} \oint_{\vert z\vert =1} e^{\tau(z-1-\nu\log z)}dz.
$$
Call $f(z)=(z-1)-\nu\log z$ and observe that $f'(z)=1-\nu/z$ and $f''(z)=\nu/z^2.$ So $f$ has a critical point at $z_c=\nu$ and $f''(z_c)=1/\nu.$ To study the growth of the integral we should look at $\mathrm{Re} f(z)$ and deform contours to go through the critical point.

\begin{figure}[ht]
\begin{center}
\includegraphics[scale=.6]{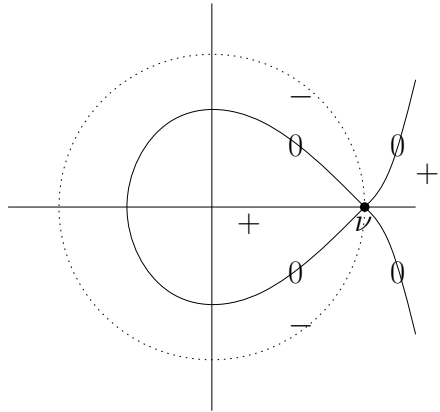}
\end{center}
\caption{Contour plot of $\mathrm{Re} f(z) - f(z_c)$ so dotted line has maximum at $\nu$ and is monotonically decreasing away.}
\end{figure}

\begin{exercise}
Show that $\mathrm{Re} f(\nu e^{i\theta})$ is strictly decreasing on $\theta\in [0,\pi].$
\end{exercise}

This implies that as $\tau\to \infty,$ the integral localizes to its value at the critical point. Since $f(\nu)=\nu-1-\nu\log\nu$ this means
$$
\gamma_1(\nu) = \nu - 1- \nu\log\nu.
$$
Consider the $k=2$ case:
$$
\E\big[z(\tau,\nu\tau)^2\big] = \frac{1}{(2\pi i)^2} \oint\oint \frac{z_1-z_2}{z_1-z_2-1} e^{f(z_1)+f(z_2)}dz_1dz_2,
$$
with nested contours for $z_1$ and $z_2$.
In order to deform both contours to the steepest descent curve we must deform $z_1$ to $z_2.$ Doing this we encounter a pole at $z_1=z_2+1.$ The residue theorem implies
\begin{align*}
\E\big[z(\tau,\nu\tau)^2\big]
&= \frac{1}{(2\pi i)^2} \oint\!\!\!\!\!\oint_{\vert z_1\vert=\vert z_2\vert =1 } \frac{z_1-z_2}{z_1-z_2-1} e^{f(z_1)+f(z_2)}dz_1dz_2 \\
&\qquad\qquad\qquad\qquad+ \frac{1}{2\pi i}\oint e^{f(z_2+1)+f(z_2)}dz_2.
\end{align*}
We must study growth of each term. The first term behaves like $e^{2f(\nu)\tau}$ while the second term requires closer consideration.

Set $\tilde{f}(z) = f(z) + f(z+1),$ then $\tilde{f}'(z) = 2 - \nu/z - \nu/(z+1)$ and so $\tilde{f}$ has a critical point at $z_c = \tfrac{1}{2} (\nu - 1 + \sqrt{\nu^2+1}).$  The plot of $\mathrm{Re} \tilde{f}(z) - \tilde{f}(z_c)$ looks quite similar to those for $f$, and hence the steepest descent proceeds similarly.

So the second term is like $e^{\tilde{f}(z_c)\tau}.$ Comparing $2f(\nu)$ to $\tilde{f}(\tfrac{1}{2}(\nu - 1 + \sqrt{\nu^2+1}))$ we find the second is larger, hence determines the growth of $\E\big[z(\tau,\nu\tau)^2\big].$

\begin{theorem}[\cite{BCLyapunov}]
The $k$-th Lyapunov exponent for the semi-discrete SHE are given by
$$
\gamma_k(\nu) = H_k(z_{c,k}),
$$
where
$$
H_k(z) = \frac{k(k-3)}{2} + kz - \nu\log\left( \prod_{i=0}^{k-1} (z+1) \right)
$$
and $z_{c,k}$ is the unique solution to $H'_k(z)=0,z\in (0,\infty).$
\end{theorem}
\begin{proof}[Proof sketch]
For general $k$ case there are many terms to consider in deforming contours. However, there is an argument (which I won't make here) that shows its the term coming from taking the residue at
$$
z_1 = z_k+k-1, \quad z_2=z_k+k-2,  \ldots, z_{k-1}= z_k+1
$$
which matters. In terms of exponential terms this yields $e^{\tau H_k(z_k)}$ and steepest descent analysis gives the result.
\end{proof}

Let us record $\tilde{\gamma}_1,$ to be proved later.

\begin{theorem}[\cite{OConMor}]
The almost sure Lyapunov exponent is given by
$$
\tilde{\gamma}_1(\nu) = - \frac{3}{2} + \inf_{s>0} \big(s-\nu\psi(s)\big)
$$
where $\psi(s)= (\log\Gamma)'(s)$ is the digamma function.
\end{theorem}

By comparing these exponents we see (in very high resolution) the intermittency behavior of this model. We also see that the $k$--moments grow like $e^{ck^2}$ so they do not determine the distribution of $z$.

We finish this section by showing how nested contour integrals (at the $q$--level) expand into residual subspaces as all contours are deformed to that of $z_k$. This will also be important in studying the $q$-moment generating function in the next section.

\begin{theorem}\label{unnestthm} Given nested contours $\gamma_1,\ldots,\gamma_k$ and a function $F(z_1,\ldots,z_k)$ which is analytic between the nested contours $\gamma_1,\ldots,\gamma_k$ we have
\begin{align}\label{Star9}
\begin{split}
&\oint_{\gamma_1} \frac{dz_1}{2\pi i} \cdots \oint_{\gamma_k} \frac{dz_k}{2\pi i} \prod_{1\leq A<B\leq k} \frac{z_A-z_B}{z_A-qz_B} F(z_1,\ldots,z_k) \\
&\qquad\qquad= \sum_{\lambda \vdash k} \underbrace{\oint_{\gamma_k}\cdots\oint_{\gamma_k}}_{\ell(\lambda)} d\mu_{\lambda}(\vec{w}) E^q(\vec{w}\circ\lambda)
\end{split}
\end{align}
where $\lambda\vdash k$ means $\lambda$ partitions $k$, a partition $\lambda=(\lambda_1\geq \lambda_2\geq\ldots \geq \lambda_{\ell(\lambda)} > 0)$ with $\lambda_i\in \mathbb{Z}_{\geq 0}, \sum \lambda_i=k,$ and $m_i=\left| \{j:\lambda_j=i\}\right|$, the {\rm(}complex valued{\rm)} measure
\begin{align}\label{dmu}
\begin{split}
d\mu_{\lambda}(\vec{w}) = \frac{(1-q)^k(-1)^k q^{-k(k-1)/2}}{m_1!m_2!\cdots} &\det \left[\frac{1}{w_i q^{\lambda_i} - w_j}\right]_{i,j=1}^{\ell(\lambda)}\\
&\qquad\times\prod_{j=1}^{\ell(\lambda)} w_j^{\lambda_j} q^{\lambda_j(\lambda_j-1)/2} \frac{dw_j}{2\pi i},
\end{split}
\end{align}
and the function
$$
E^q(z_1,\ldots,z_k) = \sum_{\sigma \in S_k} \prod_{1\leq B<A\leq k} \frac{z_{\sigma(A)}-qz_{\sigma(B)}}{z_{\sigma(A)}-z_{\sigma(B)}} F(z_{\sigma(1)},\ldots, z_{\sigma(k)})
$$
with
$$
\vec{w}\circ\lambda = (w_1,qw_1,\ldots,q^{\lambda_1-1}w_1,w_2,qw_2,\ldots,q^{\lambda_2-1}w_2,\ldots,w_\ell,qw_\ell,\ldots,q^{\lambda_l-1}w_\ell)
$$
where $\ell=\ell(\lambda)$.
\end{theorem}
\begin{proof}
Let us consider an example, $k=3,$ to get an idea of residue expansion structure. The cross term (which gives poles) is
$$
\frac{z_1-z_2}{z_1-qz_2}\frac{z_1-z_3}{z_1-qz_3}\frac{z_2-z_3}{z_2-qz_3}.
$$
As $\gamma_2$ deforms to $\gamma_3$ we cross poles at $z_2=qz_3.$ We can either take the residue or the integral.
\begin{itemize}
\item
Residue: Cross term becomes $\tfrac{z_1-z_3}{z_1-qz_3}z_3(q-1)$ and only the $z_1,z_2$ integrals survive. As $\gamma_1$ deforms to $\gamma_3$ it crosses a pole at $z_1=q^2z_3$. Again we can either pick the residue or integral.

\item
Integral: Cross term remains the same. Now shrink $\gamma_1$ to $\gamma_2$. It crosses poles at $z_1=qz_2$ and $z_1=qz_3$. We can pick either of these residues or the integral.
\end{itemize}
In total, get five terms for various residual subspaces: $z_1=qz_2$ and $z_2=qz_3$; $z_1,z_2=qz_3$; $z_1=qz_2,z_3;$ $z_1=qz_3,z_2;$ $z_1,z_2,z_3$.

In general we index our residue expansion via partitions $\lambda\vdash k$ and strings
\begin{align*}
& i_1<i_2<\ldots<i_{\lambda_1}, \\
I = \quad \quad \quad  & j_1<j_2<\ldots<j_{\lambda_2}, \\
&  \quad \quad \vdots
\end{align*}
where these indices are disjoint and union to $\{1,\ldots,k\}$. Each row corresponds with taking residues at $z_{i_1}=qz_{i_2},z_{i_2}=qz_{i_3},\ldots,z_{j_1}=qz_{j_2},z_{j_2}=q z_{j_3},\ldots, z_{j_{\lambda_2-1}}=q z_{j_{\lambda_2}},\ldots$. This is written as ${\mathrm{Res}^q_I} f(z_1,\ldots,z_k)$ and the output is a function of $z_{i_{\lambda_1}},z_{j_{\lambda_2}},\ldots.$

For example, for $k=3$ and $\lambda=(2,1)$ we can have $I\in \left\{{ 1<2 \atop 3},{1<3 \atop 2}, {2<3 \atop 1} \right\} $.

Calling $S(\lambda)$ the set of all $I$'s corresponding to $\lambda, $ we thus see that
\begin{align*}
\text{LHS}\eqref{Star9} &= \sum_{\lambda \vdash k} \frac{1}{m_1!m_2!\ldots}\\
&\qquad\times \sum_{I\in S(\lambda)}\underbrace{\oint_{\gamma_k} \frac{dz_1}{2\pi i} \cdots \oint_{\gamma_k}}_{\ell(\lambda)}\frac{dz_\ell}{2\pi i}\mathrm{Res}_I^q \left( \prod_{1\leq A<B\leq k} \frac{z_A-z_B}{z_A-qz_B} F(z_1,\ldots,z_k)\right)
\end{align*}
where the $m_i!$'s come from symmetries of $\lambda$ and multi--counting of residue subspaces arising from that. For each $I$ we can relabel the $z$ variables as
\begin{align*}
(z_{i_1}, \ldots, z_{i_{\lambda_1}}) &\mapsto (y_{\lambda_1},\ldots, y_1) \\
(z_{j_1},\ldots, z_{j_{\lambda_2}}) &\mapsto (y_{\lambda_1+\lambda_2},\ldots,y_{\lambda_1+1})
\end{align*}
and call $\sigma\in S_k$ the permutation taking $z_i\mapsto y_{\sigma(i)}$ and call the last $y$'s $w_j=y_{\lambda_1+\ldots+\lambda_{j-1}+1}$.  This change of variables puts all residual spaces with given $\lambda$ into the form
\begin{align*}
y_{\lambda_1}=qy_{\lambda_1-1}, &\ldots, y_2=qy_1 \\
y_{\lambda_1+\lambda_2} = qy_{\lambda_1+\lambda_2-1}, &\ldots, y_{\lambda_1+2}=qy_{\lambda_1+1}\\
&\vdots
\end{align*}
We denote the residue on this subspace $\mathrm{Res}_{\lambda}^q$. Thus the left-hand side of \eqref{Star9} equals
\begin{align*}
&\sum_{\lambda \vdash k} \frac{1}{m_1!m_2!\ldots}\\
&\qquad\times \sum_{\sigma \in S_k} \underbrace{\oint_{\gamma_k}\frac{dw_1}{2\pi i} \cdots \oint_{\gamma_k}}_{\ell(\lambda)}\frac{dw_\ell}{2\pi i} \mathrm{Res}_{\lambda}^q \left( \prod_{1\leq A<B\leq k} \frac{y_{\sigma(A)}-y_{\sigma(B)}}{y_{\sigma(A)}-qy_{\sigma(B)}} F(y_{\sigma(1)},\ldots,y_{\sigma(k)})\right).
\end{align*}
Notice: not all $\sigma\in S_k$ arise from $I$'s, but it is easy to see that those which do not have residue $0$. Finally, rewrite
$$
\prod_{1\leq A<B\leq k}  \frac{y_{\sigma(A)}-y_{\sigma(B)}}{y_{\sigma(A)}-qy_{\sigma(B)}}  = \prod_{A\neq B} \frac{y_A-y_B}{y_A-qy_B} \prod_{1\leq B<A\leq k} \frac{y_{\sigma(A)}-qy_{\sigma(B)}}{y_{\sigma(A)}-y_{\sigma(B)}}.
$$
The $S_k$--symmetric part has no poles in the residue subspaces and can be factored, leaving
$$
\text{LHS}\eqref{Star9}= \sum_{\lambda\vdash k} \frac{1}{m_1!m_2!\cdots } \oint_{\gamma_k} \frac{dw_1}{2\pi i} \cdots \oint_{\gamma_k} \frac{dw_{\ell}}{2\pi i} \mathrm{Res}_{\lambda}^q\left( \prod_{A\neq B} \frac{y_A-y_B}{y_A-qy_B}\right)  E^q(\vec{w}\circ\lambda).
$$
Regarding $\mathrm{Res}_{\lambda}^q$ as a function of $w$'s one checks  that this yields $d\mu_{\lambda}(\vec{w})$.
\end{proof}

\section[From moments to the Fredholm determinant]{From moments to the $q$--Laplace transform Fredholm determinant}\label{L8}
\index{Fredholm determinant}%
We will start by applying the residue expansion theorem to our nested contour integral formula for $\E^{\text{step}}[q^{k(x_n(t)+n)}]$, i.e. taking all $n_i \equiv n$.

\begin{corollary}
For step initial data $q$--TASEP, $k\geq 1$ and $n\geq 1$,
\begin{align*}
&\E^{\text{step}}\big[q^{k(x_n(t)+n)}\big] \\
&\quad= k_q! \sum_{\la\vdash k} \frac{(1-q)^k}{m_1!m_2!\ldots} \oint_\gamma \frac{dw_1}{2\pi i} \cdots \oint_\gamma \frac{dw_{\ell(\la)}}{2\pi i} \det\left[ \frac{1}{w_iq^{\la_i}-w_j}\right]_{i,j=1}^{\ell(\lambda)} \prod_{j=1}^{\ell(\la)} \frac{e^{t(q^{\la_j-1})w_j}}{(w_j;q)^n_{\la_j}},
\end{align*}
where the integral occurs over $\gamma$, a small contour containing 1. Here, $(a;q)_\ell = (1-a)(1-qa)\cdots (1-q^{\ell-1} a)$ is called the $q$--Pochhammer symbol, $k_q! = (q;q)_k / (1-q)^k.$
\end{corollary}
\begin{proof}
Apply the residue expansion theorem with
$$
F(z_1,\ldots,z_k) = (-1)q^{k(k-1)/2} \prod_{j=1}^k \frac{e^{t(q-1)z_j}}{(1-z_j)^n}\frac{1}{z_j}.
$$
Observe that due to the symmetry of $F$ in the $z$'s
$$
E^q(z_1,\ldots,z_k) = F(z_1,\ldots,z_k) \cdot \sum_{\sigma\in S_k} \prod_{1\leq B<A\leq k} \frac{z_{\sigma(A)}-qz_{\sigma(B)}}{z_{\sigma(A)}-z_{\sigma(B)}} = F(z_1,\ldots,z_k)\cdot k_q!
$$
and the rest comes from simplifications and telescoping in the geometric substitutions.
\end{proof}

\begin{exercise}
Prove that
$$
\sum_{\sigma \in S_k} \prod_{B<A} \frac{z_{\sigma(A)}-qz_{\sigma(B)}}{z_{\sigma(A)}-z_{\sigma(B)}}=k_q!
$$
\end{exercise}

The formula for $\E\big[q^{k(x_n(t)+n)}\big]$ can be written as
$$
\E\big[q^{k(x_n(t)+n)}\big] = k_q! \sum_{l=1}^{\infty} \frac{1}{l!} \sum_{\substack{\lambda_1,\ldots\lambda_{\ell}=1\\ \lambda_1+\cdots +\lambda_l=k}}^\infty \oint_\gamma \frac{dw_1}{2\pi i}\cdots \oint_\gamma \frac{dw_\ell}{2\pi i} \det \left[ K_1(\la_i,w_i;\la_j,w_j) \right]_{i,j=1}^l,
$$
where
$$
K_1(\la,w,\la',w') = \frac{(1-q)^{\la} e^{(q^{\la}-1)t w} (w;q)_\la^{-n}}{wq^{\la}-w'}.
$$

Notice how the symmetry factor $\tfrac{1}{m_1!m_2!\ldots}$ was replaced by $\frac{1}{l!}$ due to the unordering of the $\la_1,\ldots,\la_k$.

In principle, since the moments of $q^{x_n(t)+n}$ determine its distribution, so too will any moment generating function. We will consider one well suited to the formulas. For $\zeta\in \mathbb{C}$ define (this is only consistent for small $\vert \zeta\vert$)
\begin{align*}
G(\zeta) &:= \sum_{k\geq 0} \E\big[q^{k(x_n(t)+n)}\big]\frac{\zeta^k}{k_q!}\\
&= 1 + \sum_{l=1}^{\infty} \frac{1}{l!} \sum_{\la_1=1}^{\infty} \cdots \sum_{\la_l=1}^{\infty} \oint \frac{dw_1}{2\pi i} \cdots \oint \frac{dw_\ell}{2\pi i} \det\left[ \tilde{K}_{\zeta}(w_i,\la_i;w_j,\la_j)\right]\\
& =: \det(1+\tilde{K}_{\zeta})_{L^2(\mathbb{Z}_{\geq 0} \times \gamma)},
\end{align*}
where
$$
\tilde{K}_{\zeta}(\la,w,\la',w') = \frac{\zeta^{\la}(1-q)^{\la} e^{(q^{\la}-1)t w} (w;q)_\la^{-n}}{wq^{\la}-w'}.
$$

Notice that the kernel does not depend on $\lambda'$, so the summations can be brought inside yielding
$$
1 + \sum_{l=1}^{\infty} \frac{1}{l!} \oint_{\gamma} \frac{dw_1}{2\pi i} \cdots \oint_{\gamma} \frac{dw_\ell}{2\pi i} \cdot \det\big[\tilde{K}_{\zeta}(w_i,w_j)\big]_{i,j=1}^l
$$
with kernel
$$
\tilde{K}_{\zeta}(w,w') = \sum_{\lambda=1}^{\infty} g(q^{\la}) \big((1-q)\zeta\big)^{\lambda}
$$
and
$$
g(q^{\la}) = \frac{e^{(q^{\la}-1)tw}\left( \frac{ (q^{\la}w;q)_{\infty} }{(w;q)_{\infty}} \right)^n}{wq^{\la}-w'}.
$$
The reason for rewriting in terms of $g$ is that it is now seen to be an analytic function in $\la$ (previously just in $\mathbb{Z}_{\geq 0}$) away from its poles. We could stop here, but we will do one more manipulation to $\tilde{K}_{\zeta}$ and replace the summation by an integral using a Mellin--Barnes representation. This is done with an eyes towards $q\nearrow 1$ asymptotics for which the sum of termwise limits in the sum defining $\tilde{K}_{\zeta}$ fails to be convergent, while the entire sum does have a nice limit. Going to contours away from $\mathbb{Z}_{\geq 0}$ allows us to show this and study the limit.

\begin{lemma} For $g$ satisfying the below conditions and $\zeta$ such that $\left| \zeta\right|<1, \zeta\in \mathbb{C}\backslash \mathbb{R}_+$,
$$
\sum_{n=1}^{\infty} g(q^n)(\zeta)^n = \int_{C_{\infty}} \frac{\pi}{\sin(-\pi s)} (-\zeta)^s g(q^s) \frac{ds}{2\pi i}.
$$
The function $g$ and contour $C_{\infty}$ must satisfy: The left--hand--side is convergent and the right--hand--side must be able to be approximated by integrals over $C_k$ contours which enclose $\{1,\ldots, k\}$ and no other singularities of the right--hand--side integrand and whose symmetric difference from $C_{\infty}$ has integral going to zero as $k$ goes to infinity.
\end{lemma}

\begin{exercise}
Prove this using residue calculus.
\end{exercise}

Applying the lemma brings us to:
\begin{theorem}
For $\left| \zeta\right|$ small enough,
$$
G(\zeta) = \sum_{k=0}^{\infty} \E\big[q^{k(x_n(t)+n)}\big] \frac{\zeta^k}{k_q!} = 1 + \sum_{l=1}^{\infty} \frac{1}{l!} \oint_{\gamma} \frac{dw_1}{2\pi i} \cdots \oint_{\gamma} \frac{dw_\ell}{2\pi i} \det\left[K_{\zeta}(w_i,w_j) \right]_{i,j=1}^l
$$
with
$$
K_{\zeta}(w,w') = \int_{-i\infty+ 1/2}^{i\infty + 1/2} \frac{\pi}{\sin(\pi s)} \big(-(1-q)\zeta\big)^s \frac{e^{(q^{s}-1)tw}\left( \frac{ (q^{s}w;q)_{\infty} }{(w;q)_{\infty}} \right)^n}{wq^{s}-w'}\frac{ds}{2\pi i}.
$$
\end{theorem}
The important fact is that $t,n,\zeta$ come into this in a very simple way! So, $G(\zeta)$ should contain the distribution of $x_n(t)+n$, but how to get it out? We need an inverse transform of sorts. For this we will go a little deeper into the theory of $q$--series.

In 1949 Hahn \cite{Hahn} introduced two $q$--deformed exponential function:
$$
e_q(x) := \frac{1}{((1-q)x;q)_{\infty}}, \quad E_q(x) = (-(1-q)x;q)_{\infty},
$$
where recall that $(a;q)_{\infty}=(1-a)(1-qa)\cdots$.

\begin{exercise}
Show pointwise convergence of $e_q(x),E_q(x)$ to $e^x$ as $q\nearrow 1$.
\end{exercise}

We will focus on $e_q(x)$. This has a Taylor series expansion for $\vert x\vert$ small
$$
e_q(x) = \sum_{k=0}^{\infty} \frac{x^k}{k_q!},
$$
which is a special case of the $q$--Binomial theorem \cite{AAR}.

\begin{exercise}
For $\vert x\vert<1,$ $\vert q\vert<1$ prove the identity
$$
\sum_{k=0}^{\infty} \frac{(a;q)_k}{(q;q)_k}x^k = \frac{(ax;q)_{\infty}}{(x;q)_{\infty}}.
$$
\end{exercise}

This, along with the fact that $q^{x_n(t)+n}\leq 1$ implies that for $\left| \zeta \right|$ small enough,
$$
\E\left[ e_q(\zeta q^{x_n(t)+n})\right] = \sum_{k=0}^{\infty} \E\left[ q^{k(x_n(t)+n)}\right]\frac{\zeta^k}{k_q!} = G(\zeta) = \det(I+K_{\zeta})_{L^2(\gamma)}.
$$
The interchange of expectation and summation is completely justified for small $\left| \zeta \right|$ and then the left--hand--side and right--hand--side are seen to be analytic in $\zeta \in \mathbb{C}\backslash\{\mathbb{R}_+\}$ thus allowing extension of the equality.

The left--hand--side is called the $e_q$--Laplace transform of $q^{x_n(t)+n}$ and the fact that it is given by a Fredholm determinant \cite{Lax} is quite miraculous and as of yet, not so well understood.

\begin{theorem}
For $\zeta \in \mathbb{C}\backslash \{\mathbb{R}_+\},q<1,t>0,n\geq 1,$
$$
\mathbb{E}\left[ e_q(\zeta q^{x_n(t)+n}) \right] =  \det(I+K_{\zeta})_{L^2(\gamma)}:=1 + \sum_{l=1}^{\infty} \frac{1}{l!} \oint_{\gamma} \cdots \oint_{\gamma} \det\big[K_{\zeta}(w_i,w_j)\big]_{i,j=1}^l
$$
\end{theorem}

We close this section by demonstrating how to invert the $e_q$--Laplace transform.

\begin{definition}
For a function $f\in \ell^1(\mathbb{Z}_{\geq 0})$ define for $\zeta \in \mathbb{C}\backslash \{q^{-m}\}_{m\geq 0}$
$$
\hat{f}(\zeta) := \sum_{n=0}^{\infty} \frac{f(n)}{(\zeta q^n;q)_{\infty}} \quad \left( = \E^f\left[ e_q\left( \frac{\zeta}{1-q}q^n\right)\right]\right).
$$
with $n$ distributed according to the measure $f$.
\end{definition}

\begin{proposition}
The $e_{q}$-Laplace transform is inverted by
$$
f(n) = -q^n \int (q^{n+1}\zeta;q)_{\infty} \hat{f}^q(\zeta) \frac{d\zeta}{2\pi i}
$$
with the $\zeta$ contour containing only $\zeta=q^{-m},0\leq m\leq n$ poles.
\end{proposition}
\begin{exercise}
Prove this via residues.
\end{exercise}

So we have found a rather concise (and analyzable) formula for the probability distribution of $x_n(t)+n,$ from which we can perform asymptotics. In fact, for our applications we can work with the $e_q$--Laplace transform instead of its inversion.

For instance, applying our knowledge from Theorem \ref{thmlimits} of how $q$--TASEP goes to the semi--discrete SHE 
\index{Semi-discrete stochastic heat equation (SHE)}%
\index{Stochastic heat equation (SHE)!semi-discrete}%
and performing some asymptotic analysis we prove:
\begin{theorem}[\cite{BorCor,BCF}]
For semi--discrete SHE with $z_0(n) = \mathbf{1}_{n=1}$ and for $\mathrm{Re}(u) \geq 0,$
$$
\E\left[ e^{-ue^{\frac{3\tau}{2}z(\tau,n)}} \right] = 1+ \sum_{l=1}^{\infty} \frac{1}{l!} \oint_{\gamma} \frac{dv_1}{2\pi i} \cdots \oint_{\gamma} \frac{dv_l}{2\pi i} \det\left[ K_u(v_i,v_j)\right]_{i,j=1}^l
$$
with contour $\gamma$ a small circle around 0 and
\begin{align*}
K_u(v,v') &= \int_{-i\infty + 1/2}^{i\infty + 1/2} \frac{\pi}{\sin(-\pi s)} \frac{g(v)}{g(v+s)}\frac{ds}{v+s-v'}\\
g(z) &= \big(\Gamma(z)\big)^nu^{-z} e^{-\tau z^2/2}.
\end{align*}
\end{theorem}

Note that had we tried to prove this theorem directly from our moment formulas we would have failed. The moment generating function is a divergent series. It is only by work at the $q$--level and explicitly summing the series that we could then take $q\nearrow 1$ and prove this result.

Here is one application:
\begin{theorem} [\cite{BorCor,BCF}]
For all $\nu>0$\index{Tracy-Widom distribution}%
$$
\lim_{\tau \rightarrow \infty} \mathbb{P}\left( \frac{\log z(\tau,\nu\tau) - \tau\tilde{\gamma}_1(\nu)}{d(\nu)\tau^{1/3}}\leq s \right) = F_{\text{GUE}}(s),
$$
where $d(\nu) = \big(-\nu\psi''(s(\nu))/2\big)^{1/3}$ with $s(\nu) = \mathrm{arg}\inf_{s>0} \big(s-\nu\psi(s)\big).$ Here, $\psi(s)=(\log\Gamma)'(s)$ is the digamma function and, as before,
$$
\tilde{\gamma}_1(\nu) = -\frac{3}{2} + \inf_{s>0}(s-\nu\psi(s))
$$
is the almost sure Lyapunov exponent. \index{Lyapunov exponent}%
\end{theorem}

This describes (in the parabolic Anderson model)\index{Parabolic Anderson Model}%
the typical mass density fluctuations or the particle location fluctuations for a continuous space limit of $q$--TASEP and is consistent with the KPZ universality class belief.

Finally, utilizing Theorem \ref{Thmoy}, one can prove
\begin{theorem}[\cite{ACQ,BCF}]
Consider $z$, the solution to the SHE with $z_0(x)=\delta_{x=0}$. For all $S$ with $\text{Re}(S)\geq 0$,
$$\E\Big[e^{-S e^{t/24} z(t,0)}\Big]  =  \det(1 - K_S)_{L^2(\R_+)}$$
with kernel
$$K_S(\eta,\eta') = \int_{\R} dr \frac{S}{S+e^{-r (t/2)^{1/3}}} \Ai(r+\eta)\Ai(r+\eta').$$

Additionally,
$$
\lim_{t\to \infty} \mathbb{P}\left(\frac{\log z(t,0)  + t/24}{(t/2)^{1/3}} \leq r \right)  = F_{{\rm GUE}}(r)
$$
where $F_{{\rm GUE}}$ is the GUE Tracy-Widom distribution (see Appendix \ref{TATW}).
\end{theorem}

\section{$q$--Boson spectral theory}\label{secspec}
Recall that $q$--Boson particle system in $\vec{n}=(n_1 \geq \ldots \geq n_k)$ coordinates has backwards generator
$$
(L^{\text{q--Boson}}f)(\vec{n}) = \sum_{\text{cluster } i} (1-q^{c_i})(f(\vec{n}_{c_1+\cdots c_i}^-) - f(\vec{n})).
$$
Also recall the generator of $k$ free (distant) particles is
$$
(\mathcal{L}u)(\vec{n}) = \sum_{i=1}^k (\nabla_i u)(\vec{n}),
$$
where $\nabla_i f (\vec{n}) = f(n-1)-f(n)$ in $n_i$ variable.

We say that $u$ satisfies the boundary conditions if
$$
(\nabla_i - q\nabla_{i+1})u \vert_{n_i=n_{i+1}} = 0,
$$
for $1\leq i \leq k-1$.

The question we now confront is how to find the left and right eigenfunctions of $L^{\text{q--Boson}}$. We have already seen the below essentially contained in the proof of Proposition \ref{partB}.

\begin{proposition}
If $u:\Z^k\rightarrow\C$ is an eigenfunction for $\mathcal{L}$ with eigenvalue $\lambda$, and $u$ satisfies the boundary conditions, then $u$ is an eigenfunction of $L^{\text{q--Boson}}$ with eigenvalue $\lambda$.
\end{proposition}

In order to find such eigenfunctions we will use another idea going back to Bethe \cite{Bethe} in 1931 (see also \cite{LL}). This idea goes under the name of coordinate Bethe ansatz and takes the following general form. \index{Bethe ansatz}%

\begin{proposition}
Consider a space $X$. Eigenfunctions for a sum of 1d operators acting on $\{\text{functions on }X\}$
$$
(\mathcal{L}\psi)(\vec{x}) = \sum_{i=1}^k (L_i \psi)(\vec{x}), \quad \vec{x} =(x_1,\ldots,x_k) \in X^k
$$
(with $L_i$ acting in the $i$--th coordinate) that satisfy boundary conditions depending on $B$ acting on $\{\text{functions on }X^2\}$
$$
B_{i,i+1}\psi (\vec{x}) \vert_{x_i=x_{i+1}} = 0, \text{ for } 1\leq i,\leq k-1
$$
(with $B_{i,i+1}$ acting on the $i,i+1$ coordinate) can be found via:
\begin{enumerate}
\item Diagonalizaing the 1d operators $L\psi_z = \lambda_z\psi_z$ where $\psi_z:X\rightarrow \C$ and $z$ is a parameter (e.g. complex number)
\item Taking linear combinations
$$
\psi_{\vec{z}}(\vec{x}) = \sum_{\sigma \in S(k)} A_{\sigma}(\vec{z})\prod_{j=1}^k \psi_{z_{\sigma(j)}}(x_j)
$$
\item Evaluating $A_{\sigma}(\vec{z})$ as
$$
A_{\sigma}(\vec{z}) = \text{sgn}(\sigma) \prod_{a>b} \frac{S(z_{\sigma(a)},z_{\sigma(b)})}{S(z_a,z_b)}
$$
where
$$
S(z_1,z_2) = \frac{B(\psi_{z_1}\otimes \psi_{z_2})(x,x)}{\psi_{z_1}(x)\psi_{z_2}(x)}.
$$
\end{enumerate}
Then $\big(\mathcal{L} \psi_{\vec{z}}\big)\vec{x}) = \big(\sum_{i=1}^{k} \lambda_{z_i}\big) \psi_{\vec{z}}(\vec{x})$.
\end{proposition}

\begin{proof}
By Liebniz rule, it is clear that $\psi_{\vec{z}}(\vec{x})$ are eigenfunctions for the free generator. It remains to check that the choice of $A_{\sigma}$ implies that the boundary conditions are satisfied. Let $\tau_i = (i,i+1)$ act on permutations by permuting $i$ and $i+1$. It suffices to find $A_{\sigma}$ (or show that our specific choice of $A_{\sigma}$) that satisfy that for all $\sigma\in S(k)$, and all $1\leq i\leq k-1$, when $B_{i,i+1}$ is applied to
$$
T_{\sigma} + T_{\tau_i \sigma},\qquad \textrm{where} \quad T_{\sigma} =A_{\sigma}(\vec{z}) \prod_{j=1}^{k} \psi_{z_{\sigma(j)}}(x_j),
$$
the result is zero whenever $x_i=x_{i+1}$. Indeed, if one sums the above left-hand side over all $\sigma\in S(k)$, the result is $2 \big(B_{i,i+1} \psi_{\vec{z}}\big)(\vec{x})$ and if each summand can be made to equal zero (when $x_i=x_{i+1}$), so too with the entire sum by zero.

From definitions, one sees that
$$
B_{i,i+1}T_{\sigma}\big\vert_{x_i=x_{i+1}} = S(z_{\sigma(i)},z_{\sigma(i+1)}) T_{\sigma}.
$$
So
$$
B_{i,i+1}\big(T_{\sigma} +T_{\tau_i \sigma}\big)\big\vert_{x_i=x_{i+1}} = S(z_{\sigma(i)},z_{\sigma(i+1)})T_{\sigma} + S(z_{\sigma(i+1)},z_{\sigma(i)}) T_{\tau_i \sigma}.
$$
Unwinding this, we find that in order for this sum to be zero, we must have
$$
A_{\tau_i\sigma}(\vec{z}) = - \frac{S(z_{\sigma(i+1)},z_{\sigma(i)})}{S(z_{\sigma(i)},z_{\sigma(i+1)})} A_{\sigma}(\vec{z}).
$$
The transpositions $\tau_i$ as $1\leq i\leq k-1$ varies generates $S(k)$ and hence fixing $A_{I}(\vec{z})\equiv 1$, the formula uniquely characterizes the $A_{\sigma}(\vec{z})$. This expression matches that claimed by the proposition.
\end{proof}

Let us apply this method when $L= \nabla$ and $(B g)(x,y) =(\nabla_1-q\nabla 2) g(x,y)$. We can write 1d eigenfunctions of $L$ as
$\psi_z(n) = (1-z)^{-n}$ (here $n$ replaces $x$ from the proposition). This choice of $\psi_z$ has eigenvalue $(1-q)z$ and leads to $S(z_1,z_2) = -(z_1-qz_2)$. Thus, Bethe ansatz produces left eigenfunctions for all $z_1,\ldots, z_k\in \mathbb{C}\setminus \{1\}$
$$
\psi^{\ell}_{\vec{z}}(\vec{n}) := \sum_{\sigma\in S(k)} \prod_{1\leq B<A\leq k} \frac{z_{\sigma(A)}-q z_{\sigma(B)}}{z_{\sigma(A)}-z_{\sigma(B)}} \prod_{j=1}^{k} (1-z_{\sigma(j)})^{-n_j}
$$
such that when restricted to $\vec{n}= (n_1\geq n_2\geq \cdots \geq n_k)$,
$$(L^{\text{q--Boson}}\psi^{\ell}_{\vec{z}})(\vec{n}) = \sum_{j=1}^{k} (1-q)z_j \, \phi^{\ell}_{\vec{z}}(\vec{n}).$$

As a brief remark, notice that these are eigenfunctions for any choices of $z_j\in \mathbb{C}\setminus \{1\}$. If instead of working with the $q$-Boson process on $\mathbb{Z}$ we considered a periodic (or other boundary condition) portion of $\mathbb{Z}$, then in order to respect the boundary conditions of the finite lattice we would need to impose additional conditions on the $\vec{z}$. These restrictions are known as the Bethe ansatz equations and make things much more involved. It is only for solutions $\vec{z}$ of these algebraic equations that $\psi^{\ell}_{\vec{z}}(\vec{n})$ are eigenfunctions (on the finite lattice).

Having identified left eigenfunctions, we come now to the question of right eigenfunctions. Had our operator $L^{\text{q--Boson}}$ been self-adjoint, the left and right eigenfunctions would coincide and we could decompose and recompose functions with respect to this single set of functions. It is clear that $L^{\text{q--Boson}}$ is not self-adjoint, however it is not too far off, as it enjoys the property of being PT-invariant (in fact, this property is shared by all totally asymmetric zero range processes, provided their jump rate $g(k)$ is reasonable). PT-invariance means that the process generated by $L^{\text{q--Boson}}$ is invariant under joint space reflection and time inversion.

The product invariant measure $\mu$ for the $q$-Boson process in the $\vec{y}$ variables have one point marginal
$$\mu_{\alpha}(y_0) = \mathbf{1}_{y_0\geq 0} \frac{ \alpha^{y_0}}{g(1)g(2)\cdots g(y_0)},$$
where $\alpha>0$ controls the overall density of the invariant measure, and $g(k) = 1-q^k$ is the $q$-Boson jump rate. Time reversal corresponds with taking the adjoint of $L^{\text{q--Boson}}$ in $L^2(\vec{y},\mu_{\alpha})$ (the choice of $\alpha$ does not, in fact, matter here). Then PT-invariance amounts to the fact that one readily shows \cite{BCPS13} that
$$
L^{\text{q--Boson}} = P \big(L^{\text{q--Boson}}\big)^* P^{-1}
$$
where $\big(Pf\big)(\vec{y}) = f\big(\{y_{-i}\}_{i\in \Z}\big)$ is the space reflection operator (clearly $P=P^{-1}$).

PT-invariance can be written in terms of the $\vec{n}$ variables and using matrix notation as
$$
L^{\text{q--Boson}} = \big(PC_q\big) \big(L^{\text{q--Boson}}\big)^{{\rm transpose}} \big(PC_q\big)^{-1}
$$
where now $\big(Pf\big)(n_1,\ldots, n_k) = f(-n_k,\ldots, -n_1)$ and $C_q$ is the multiplication operator with
$$C_q(\vec{n}) =(-1)^k q^{-\frac{k(k-1)}{2}} \prod_{\text{cluster } i} \frac{(q;q)_{c_i}}{(1-q)^{c_i}}.$$
Note that we could have defined $C_q(\vec{n})$ as any function depending only on $k$, times the product over clusters of the $(q;q)_{c_i}$ terms. The choice above will be well suited to avoid messy constants in what follows.

Returning to the matter of right eigenfunctions, it is clear from PT-invariance that applying $\big(PC_q\big)^{-1}$ to left eigenfunction, produces right eigenfunctions. Thus, we define right eigenfunctions
$$
\psi^{r}_{\vec{z}}(\vec{n}) := \sum_{\sigma\in S(k)} \prod_{1\leq B<A\leq k} \frac{z_{\sigma(A)}-q^{-1} z_{\sigma(B)}}{z_{\sigma(A)}-z_{\sigma(B)}} \prod_{j=1}^{k} (1-z_{\sigma(j)})^{n_j}
$$
which satisfy
$$
\Big(\big(L^{\text{q--Boson}}\big)^{{\rm transpose}} \psi^{r}_{\vec{z}}\Big)(\vec{n}) = \sum_{j=1}^{k} (1-q)z_j \psi^{r}_{\vec{z}}(\vec{n}).
$$
Note that $\psi^{r}_{\vec{z}}(\vec{n}) = q^{-\frac{k(k-1)}{2}} \big(PC_q\big)^{-1} \psi^{\ell}_{\vec{z}}(\vec{n})$.

Having defined left and right eigenfunctions, it remains to demonstrate how to diagonalize $L^{\text{q--Boson}}$ with respect to them, with the ultimate goal of solving $\frac{d}{dt} h(t;\vec{n}) = L^{\text{q--Boson}} h(t;\vec{n})$ for arbitrary initial data. We proceed now by defining a direct and inverse Fourier type transform.

\begin{definition}
Fix the following spaces of functions:
\begin{align*}
\mathcal{W}^k &= \Big\{f:\big\{\vec{n} = (n_1\geq n_2\geq \cdots \geq n_k)\} \to \mathbb{C}\, \text{of compact support}\Big\},\\
\mathcal{C}^k &= \mathbb{C}\Big[(z_1-1)^{\pm1},\ldots,(z_k-1)^{\pm 1}\Big]^{\text{Sym}}.
\end{align*}
In words, $\mathcal{W}^k$ are all functions of $\vec{n}$ to $\mathbb{C}$ of compact support in $\vec{n}$ and $\mathcal{C}^k$ are all symmetric Laurant polynomials in the variables $(z_1-1)$ through $(z_k-1)$.
We may define bilinear pairings on these spaces so that for $f,g\in \mathcal{W}^k$ and $F,G\in \mathcal{C}^k$,
\begin{align*}
\big\langle f,g \big\rangle_{\mathcal{W}} &= \sum_{n_1\geq \cdots \geq n_k} f(\vec{n})g(\vec{n})\\
\big\langle F,G \big\rangle_{\mathcal{C}} &= \oint\cdots \oint d\mu_{(1)^k}(\vec{w}) \prod_{j=1}^{k} \frac{1}{1-w_j} F(\vec{w}) G(\vec{w}),
\end{align*}
where the integrals are over circles centered at the origin or radius exceeding one, and the notation $d\mu_{\lambda}(\vec{w})$ (here $(1)^k$ is the partition with $k$ ones) is recalled from (\ref{dmu}).

Define the direct transform $\mathcal{F}: \mathcal{W}^k \to \mathcal{C}^k$ and (candidate) inverse transform $\mathcal{J}: \mathcal{C}^k \to \mathcal{W}^k$ as
\begin{align*}
\big(\mathcal{F} f\big) (\vec{n}) &=\big\langle f,\psi^{r}_{\vec{z}} \big\rangle_{\mathcal{W}}\\
\big(\mathcal{J} G\big) (\vec{n}) &=\big\langle \psi^{\ell}(\vec{n}),G \big\rangle_{\mathcal{W}}.
\end{align*}
In the second line we have used $\psi^{\ell}(\vec{n})$ to represent the function which maps $\vec{z}$ to $\psi^{\ell}_{\vec{z}}(\vec{n})$.
\end{definition}

The operator $\mathcal{J}$ can be written in two alternative ways.  The first is given by
$$
\big(\mathcal{J} G\big) (\vec{n}) = \frac{1}{(2\pi i)^k}
\oint\cdots\oint \prod_{1\leq A<B\leq k} \frac{z_A-z_B}{z_A-qz_B} \prod_{j=1}^k \frac{1}{(1-z_j)^{n_j+1}} dz_j
$$
with contours such that the $z_A$ contour contains $\{qz_B\}_{B>A}$ and $1$. This equivalence follows by the fact that one can deform such contours to all lie on a large circle centered at zero of radius exceeding one. Then, since all contours lie upon the same circle, one can symmetrize the integrand and after an application of the Cauchy determinant formula, we recover the initial expression for $\mathcal{J}$. The second expression comes from unnesting the contours, but onto a single small contour around one. This is accomplished by applying Theorem \ref{unnestthm}, and hence we find that
$$
\big(\mathcal{J} G\big) (\vec{n}) = \sum_{\lambda\vdash k} \frac{1}{m_1!m_2!\cdots} \oint\cdots\oint d\mu_{\lambda}(\vec{w}) \prod_{j=1}^{\ell(\lambda)} \frac{1}{(w_j;q)_{\lambda_{j}}} \, \psi^{\ell}_{\vec{w}\circ \lambda}(\vec{n}) G(\vec{w}\circ \lambda),
$$
where the $\ell(\lambda)$ integrals occur upon a single contour which contains 1 and has small enough radius so the image under multiplication by $q$ is outside the contour.

\begin{theorem}[\cite{BCPS13}]
On the spaces $\mathcal{W}^k$ and $\mathcal{C}^{k}$ the operators $\mathcal{F}$ and $\mathcal{J}$ are mutual inverses (i.e. $\mathcal{J}\mathcal{F}$ is the identity on $\mathcal{W}^k$ and $\mathcal{F}\mathcal{J}$ is the identity in $\mathcal{C}^k$). Consequently, the left and right eigenfunctions are biorthogonal so that
$$
\big\langle \psi^{\ell}_{\bullet}(\vec{m}), \psi^{r}_{\bullet}(\vec{n})\big\rangle_{\mathcal{C}} = \mathbf{1}_{\vec{m}=\vec{n}}
$$
and
$$
\big\langle \psi^{\ell}_{\vec{z}}(\bullet), \psi^{r}_{\vec{w}}(\bullet)\big\rangle_{\mathcal{W}} = \frac{1}{k!}\prod_{1\leq A\neq B \leq k} \frac{z_A-qz_B}{z_A-z_B} \prod_{j=1}^{k} \frac{1}{1-z_j} \, \det\big[\delta(z_i-w_j)\big]_{i,j=1}^{k}
$$
where this last equality is to be understood in a weak (or integrated) sense.
\end{theorem}

At the end of this section we will prove part of this theorem, that $\mathcal{J}\mathcal{F}$ is the identity. A generalization of the above theorem (which in fact admits easier proofs) is provided in \cite{BCPS14}. But first, let us apply this theorem to solve the $q$-Boson backward equation.

\begin{corollary}
For $h_0\in \mathcal{W}^k$, the solution to
$$
\frac{d}{dt} h(t,\vec{n}) = \big(L^{\text{q--Boson}} h\big)(t,\vec{n})
$$
with initial data $h_0$ is given by
\begin{align*}
h(t,\vec{n}) &= \mathcal{J}\big(e^{t(q-1)(z_1+\cdots+z_k)} \mathcal{F} h_0\big)(t,\vec{n})\\
 &= \frac{1}{(2\pi i)^k} \oint \cdots \oint \prod_{1\leq A<B\leq k} \frac{z_A-z_B}{z_A-qz_B} \prod_{j=1}^{k} e^{t(q-1)z_j}{(1-z_j)^{n_j+1}} \big(\mathcal{F} h_0\big)(\vec{z}) d\vec{z},
\end{align*}
with nested contours.
\end{corollary}
\begin{proof}
From the theorem, $\mathcal{J}\mathcal{F}$ is the identity on $\mathcal{W}^k$. Hence,
$$
h(t,\vec{n}) = \big(e^{t L^{\text{q--Boson}}} h_0\big)(\vec{n}) = \big( e^{t L^{\text{q--Boson}}} \mathcal{J}\mathcal{F} h_0\big)(\vec{n})
$$
Using the explicit form of the transforms, we find that
$$
\big( e^{t L^{\text{q--Boson}}} \mathcal{J}\mathcal{F} h_0\big)(\vec{n}) =  e^{t L^{\text{q--Boson}}} \oint \cdots\oint d\mu_{(1)^k}(\vec{w}) \prod_{j=1}^k \frac{1}{1-w_j} \psi^{\ell}_{\vec{w}}(\vec{n}) \big(\mathcal{F} h_0\big)(\vec{z}).
$$
We now use the fact that $\psi^{\ell}_{\vec{w}}$ are the left eigenfunctions for $L^{\text{q--Boson}}$ and thus $e^{t L^{\text{q--Boson}}}$ acts diagonally with eigenvalue $e^{t(q-1)(w_1+\cdots w_k)}$. Plugging this in (and going to the nested contours as explained earlier) yields the formula in the corollary.
\end{proof}

There are two apparent limitations of this corollary. The first is that it only applies for $h_0\in\mathcal{W}^k$. In Section \ref{L6} we were concerns with initial data $h_0(\vec{n}) = \prod_{j=1}^{k} \mathbf{1}_{n_j\geq 1}$, which does not have compact support. For this initial data it is possible to extend the validity of the corollary with a little work (likely this can be done for a much wider class as well). The second limitation is more serious -- the expression $\big(\mathcal{F}h_0\big)(\vec{z})$ involves an infinite summation of initial data against eigenfunctions. However, such an infinite summation is not so useful for asymptotics. We would like to be able to evaluate such summations. Unfortunately, this may be quite hard. However, such a summation is automatic is $h_0= \mathcal{J} G_0$ for some $G_0\in \mathcal{C}^k$. In that case, $\mathcal{F} h_0 = \mathcal{F}\mathcal{J} G_0 = G_0$ by the theorem.

For initial data $h_0(\vec{n}) = \prod_{j=1}^{k} \mathbf{1}_{n_j\geq 1}$ it can be shown (via residues and extending the fact that $\mathcal{F}\mathcal{J}$ is the identity to some functions outside of $\mathcal{C}^k$) that $h_0(\vec{n}) = \big(\mathcal{J} G_0\big)(\vec{n})$ where $G_0(\vec{z}) = q^{\frac{k(k-1)}{2}} \prod_{j=1}^k \frac{z_j-1}{z_j}$. Using this we can recover the solution given in Theorem \ref{thmgenmom} and its corollary. This approach can be applied to some broader classes of initial data (cf. \cite{BCS}), though we will not pursue these here.

Let us close this section by proving that
$$
\mathcal{K} := \mathcal{J}\mathcal{F} = {\rm Identity}
$$
on $\mathcal{W}^k$.

\begin{proof}[Sketch of $\mathcal{J}\mathcal{F} = {\rm Identity}$]
PT-invariance implies that for $f,g\in \mathcal{W}^k$,
$$
\big\langle \mathcal{K} f,g \big\rangle_{\mathcal{W}} = \big\langle f, (PC_q)^{-1} \mathcal{K} (PC_q)g\big\rangle_{\mathcal{W}}.
$$
This can be shown by expanding $\mathcal{K}$ into eigenfunctions and using the relation between left and right eigenfunctions implied by PT-invariance.

In order to prove that $\mathcal{K}$ acts as the identity, it suffices (by linearity) to show that for $f(\vec{n}) = \mathbf{1}_{\vec{n}=\vec{x}}$ for some $\vec{x}$ fixed, $\big(\mathcal{K} f\big)(\vec{y}) = \mathbf{1}_{\vec{y}=\vec{x}}$. Showing that $\big(\mathcal{K} f\big)(\vec{x})=1$ involves residue calculus, and we will skip it. We will, however, show that for $\vec{y}\neq \vec{x}$, $\big(\mathcal{K} f\big)(\vec{y})=0$.

Set $g(\vec{n}) = \mathbf{1}_{\vec{n}=\vec{y}}$, then
$$
\big(\mathcal{K} f\big)(\vec{y}) = \big\langle \mathcal{K} f, g\big\rangle_{\mathcal{W}} = \frac{1}{(2\pi i)^k} \oint\cdots \oint \prod_{1\leq A<B\leq k} \frac{z_A-z_B}{z_A-qz_B} \prod_{j=1}^{k} (1-z_j)^{-y_j-1} \psi^{r}_{\vec{z}}(\vec{x}),
$$
with integration along nested contours. We wish to prove that this is zero. Consider expanding the $z_1$ contour to infinity. This can be done without crossing and poles, so evaluating that integral amounts to evaluating the residue of the integrand at infinity. By expanding the right eigenfunction via its definition, we see that the largest exponent of $(1-z_1)$ is $x_1-y_1-1$. Thus, in order that there be a residue at infinity, we should have $x_1-y_1-1\geq -1$, or in other words $x_1\geq y_1$. The implication is that if $x_1<y_1$, then $\big(\mathcal{K} f\big)(\vec{y})=0$.

But, using the PT-invariance relation
$$
\big(\mathcal{K} f\big)(\vec{y}) = \big\langle \mathcal{K} f, g\big\rangle_{\mathcal{W}} = \big\langle f, (PC_q)^{-1} \mathcal{K} (PC_q) g\big\rangle_{\mathcal{W}}.
$$
This switches the role of $x$'s and $y$'s and the same reasoning as above shows that this is zero if $y_1<x_1$. The overall implication is that for
$\big(\mathcal{K} f\big)(\vec{y})$ to be nonzero, we must have $x_1=y_1$.

In light of this deduction, assume that $x_1=y_1$. It is now quite simply to evaluate the residue as $z_1$ goes to infinity. This yields a similar expression except without the $z_1$ integration valuable. Thus, in the same manner we establish that in order that $\big(\mathcal{K} f\big)(\vec{y})$ be nonzero, $x_2=y_2$ and so on until we find that all $x_j=y_j$.
\end{proof}

\appendix

\section{White noise and stochastic integration}\label{secwhitenoise}
In one sentence, white noise $\xi(t,x),t\geq 0,x\in \R$ is the distribution valued Gaussian process with mean zero and covariance
$$
\E[\xi(t,x)\xi(s,y)] = \delta(x-y)\delta(s-t).
$$
Let us be more precise about what this means. Recall that if $g$ is a distribution and $f$ is a smooth test function, then $\int fg$ is a real number. For example, $\int f(x)\delta(x-y)dx=f(y)$. Therefore, for a smooth function $f(t,x)$ of compact support, the integral
$$
\int_{\R_+\times \R} f(t,x)\xi(t,x)dxdt
$$
is a random variable. These random variables are jointly Gaussian with mean zero and covariance
\begin{align*}
&\E\left[ \int_{\R_+\times \R} f_1(t,x)\xi(t,x)dxdt \int_{\R_+\times \R} f_2(s,y)\xi(s,y)dyds \right] \\
&= \int_{\R_+\times \R} dxdt \int_{\R_+\times \R} \E\left[ \xi(s,y)\xi(t,x)\right]f_1(t,x) f_2(s,y)dyds \\
&= \int_{\R_+\times \R} f_1(t,x)f_2(t,x)dxdt.
\end{align*}
There are many ways to construct $\xi(t,x)$. For instance, since the covariance is positive definite and symmetric, general theory of Gaussian processes implies the existence and uniqueness of this process. More concretely, one can choose any orthonormal basis $\{f_n\}$ of $L^2(\R_+\times \R)$ and independent Gaussian random variables $\{Z_n\}$ with mean zero and variance one, and define
$$
\xi(t,x) = \sum_{n=1}^{\infty} Z_n f_n(t,x).
$$
Within a suitable negative Sobolev space (namely $H_{-1-\delta,\text{loc}}(\R_+\times \R)$ for any $\delta>0$) this construction will yield a unique (in law) element whose covariance can be checked to be as desired. Details on constructing $\xi$ can be found, for instance, in \cite{J97,PT10}.

%

Assume that $\xi(t,x)$ has been constructed on some probability space $(\Omega,\mathcal{F},\mathbb{P})$.  Now let us construct stochastic integrals with respect to white noise. For non--random functions $f(t,x)$ this is not hard. If $f\in L^2(\mathbb{R}_+ \times \mathbb{R})$ then there are smooth functions $f_n$ with compact support in $\R_+\times \R$ such that
$$
\int_{\RR} \left| f_n(t,x) - f(t,x) \right|^2 dxdt \rightarrow 0.
$$
Since
$$
\E\left[ \left( \int_{\RR} (f_n(t,x) - f_m(t,x))\xi(t,x)dxdt \right)^2 \right] = \int_{\RR} \left| f_n(t,x) - f_m(t,x) \right|^2 dxdt,
$$
this means that $\int_{\RR} f_n(t,x)\xi(t,x)dxdt$ is a Cauchy sequence in $L^2(\Omega,\mathcal{F},\mathbb{P})$. Since $L^2(\Omega,\mathcal{F},\mathbb{P})$ is complete, then this Cauchy sequence has a limit, which we define to be $\int_{\RR} f(t,x)\xi(t,x)dxdt$.

For random functions, the definition of the integral is a little more complicated. As in the one dimensional case, one has to make a choice of the integral (It\^{o} vs. Stratonovich) . The construction here is essentially the standard It\^{o} integral, but only in the time variable.

Start with smooth functions $\varphi(x)$ on $\R$ with compact support. For $t>0$ we can define
$$
\int_{\RR} \mathbf{1}_{(0,t]}(s)\varphi(x)\xi(s,x)dsdx.
$$
This is a Brownian motion in $t$ with variance $\int \varphi^2(x)dx$, since it is Gaussian with mean zero and covariance
\begin{align*}
&\E\left[ \int_{\RR} \mathbf{1}_{(0,t_1]}(s)\varphi(x)\xi(x,s)dxds \int_{\RR} \mathbf{1}_{(0,t_2]}(s)\varphi(x)\xi(x,s)dxds \right] \\
&= \int_{0}^{\min(t_1,t_2)} ds \int_{\R} \varphi^2(x)dx\\
&= \min(t_1,t_2) \int \varphi^2(x)dx.
\end{align*}
(Note: one often hears the statement ``white noise is the derivative of Brownian motion'').

Let $\mathcal{F}_0=\emptyset$ and for each $t>0$ define $\mathcal{F}_t$ to be the $\sigma$--field generated by
\begin{align*}
&\Bigl\{ \int_{\RR} \mathbf{1}_{(0,s]}(u) \varphi(x) \xi(u,x) dxdu: 0\leq s\leq t,\\
&\qquad\qquad\qquad \varphi \text{ a smooth function of compact support in } \R \Bigr\}.
\end{align*}
It is clear that $\mathcal{F}_t$ is a filtration of $\mathcal{F}$, that is $\mathcal{F}_s\subseteq \mathcal{F}_t$ for $s\leq t$.

Now consider slightly more complicated functions. Let $\mathcal{S}$ be the set of functions of the form
$$
f(t,x,\omega) = \sum_{i=1}^n X_i(\omega) \mathbf{1}_{(a_i,b_i]}(t)\varphi_i(x),
$$
where $X_i$ is a bounded $\mathcal{F}_{a_i}$--measurable random variable and $\varphi_i$ are smooth functions of compact support on $\R$. For functions of this form, define the stochastic integral as
$$
\int_{\RR} f(t,x)\xi(t,x)dxdt = \sum_{i=1}^n X_i \int_{\RR} \mathbf{1}_{(a_i,b_i]}(t) \varphi_i(x) \xi(t,x)dxdt.
$$
It is easy to check that the integral is linear and an isometry from $L^2(\RR\times \Omega,\mathcal{F},\mathbb{P})$ to $L^2(\Omega,\mathcal{F},\mathbb{P})$, that is
$$
\int_{\RR} \mathbb{E}[f^2(t,x)]dxdt = \mathbb{E} \left[ \left( \int_{\RR} f(t,x)\xi(t,x)dxdt \right)^2 \right].
$$

Let $\mathcal{P}$ be the sub--$\sigma$--field of $\mathcal{B}(\RR)\times \mathcal{F}$ generated by $\mathcal{S}$. Let $L^2(\RR\times\Omega,\mathcal{F},\mathbb{P})$ be the space of square integrable $\mathbb{P}$--measurable random variables $f(t,x,\omega)$. These will be the integrators. It is important to note that these are non--anticipating in the sense that $f(t,x,\omega)$ only depends on the information $\mathcal{F}_t$ up to time $t$. This is analogous to the distinction between It\^{o} and Stratonovich integrals in the one--dimensional case. The construction of the stochastic integral will be defined through the isometry and approximation.

\begin{lemma}
$\mathcal{S}$ is dense in $L^2(\RR\times \Omega, \mathcal{F}, \mathbb{P}).$
\end{lemma}
\begin{proof}
Same as one--dimensional case.
\end{proof}

Thus, if $f\in L^2(\RR\times\Omega, \mathcal{F}, \mathbb{P})$ there exist $f_n\in \mathcal{S}$ wuch that $f_n$ converges to $f\in L^2(\RR\times\Omega, \mathcal{F}, \mathbb{P})$. By the isometry,
$$
I_n(\omega):= \int_{\RR} f_n(t,x,\omega)\xi(t,x)dxdt
$$
is a Cauchy sequence in $L^2(\Omega, \mathcal{F},\mathbb{P})$. Hence there is a limit point $I\in L^2(\Omega, \mathcal{F},\mathbb{P})$ which is defined to be the stochastic integral $\int_{\RR} f(t,x)\xi(t,x)dxdt.$ This is linear in $f$ and the It\^{o} isometry holds.

We can also define multiple stochastic integrals. Let $\Lambda_k = \{0 < t_1 < \cdots < t_k < \infty\}$.
\begin{multline*}
\int_{\Lambda_k}\int_{\R^k} f(\vec{t},\vec{x})\xi^{\otimes k}(\vec{t},\vec{x})d\vec{t}d\vec{x} \\
:= \int_{\Lambda_k}\int_{\R^k} f(t_1,\ldots,t_k,x_1,\ldots,x_k) \xi(t_1,x_1) \cdots \xi(t_k,x_k)dx_1\cdots dx_k dt_1\cdots dt_k.
\end{multline*}
This is defined inductively. For example,
\begin{multline*}
\int_{\Lambda_2} \int_{\R^2} f(t_1,t_2,x_1,x_2)\xi(t_1,x_1)\xi(t_2,x_2)dx_1dx_2dt_1dt_2 \\
= \int_0^{\infty} \int_{\R} \left[ \int_0^{t_2} \int_{\R} f(t_1,t_2,x_1,x_2)\xi(t_1,x_1)dx_1dt_1 \right] \xi(t_2,x_2)dx_2dt_2,
\end{multline*}
which is well defined because we just showed that the integrand is progressively measurable. The covariance of these multiple stochastic integrals is
$$
\E\left[ \int_{\Lambda_k} \int_{\R^k} f(\vec{t},\vec{x})\xi^{\otimes k}(\vec{t},\vec{x})d\vec{t}d\vec{x} \int_{\Lambda_j} \int_{\R^j} g(\vec{t},\vec{x})\xi^{\otimes j}(\vec{t},\vec{x})d\vec{t}d\vec{x} \right] = \langle f,g \rangle_{L^2(\Lambda_k \times \R_k)} \mathbf{1}_{j=k}.
$$
It also turns out that they span $L^2(\Omega, \mathcal{F}, \mathbb{P})$.

This defines an isometry between $\bigoplus_{k=0}^{\infty} L^2(\Lambda_k \times \R^k,\mathcal{F},\mathbb{P})$ and $L^2(\Omega, \mathcal{F}, \mathbb{P})$ by
$$
X = \sum_{k=0}^{\infty} \int_{\Lambda_k} \int_{\R^k} f_k(\vec{t},\vec{x})\xi^{\otimes k}(\vec{t},\vec{x})d\vec{t}d\vec{x}.
$$
Here $f_0$ is the constant function $\E X$. This summation is called the Wiener chaos decomposition. \index{Chaos series}%

\vspace{0.5in}

Aside: in a sense, the Wiener chaos decomposition holds more generally (that is, for Gaussian Hilbert spaces, not just for stochastic integrals with respect to white noise). If $H$ is a Gaussian Hilbert space in a probability space $(\Omega, \mathcal{F}, \mathbb{P})$ let $\overline{\mathcal{P}}_k(H)$ be the closure in $L^2(\Omega, \mathcal{F}, \mathbb{P})$ of
$$
\mathcal{P}_k(H) = \{p(\xi_1,\ldots,\xi_m): p \text{ polynomial of degree } \leq k, \xi_1,\ldots,\xi_m\in H, m<\infty\}.
$$
Let
$$
H^{:k:} = \displaystyle\overline{\mathcal{P}}_k(H) \ominus \overline{\mathcal{P}}_{k-1}(H) = \overline{\mathcal{P}}_k(H) \cap \overline{\mathcal{P}}_{k-1}(H)^{\perp}.
$$
Set $H^{:0:}$ to be the constants. The Wiener chaos decomposition is then
$$
L^2(\Omega, \mathcal{F}(H), \mathbb{P}) = \bigoplus_{k=0}^{\infty} H^{:k:}.
$$
For example, if the probability space is $(\mathbb{R}, \mathcal{B}, \gamma)$ where $d\gamma = \tfrac{1}{\sqrt{2\pi }} e^{-x^2/x}dx$ and $H=\{tx:t\in \R\}$ is a one--dimensional Gaussian Hilbert space, then $H^{:k:}$ is also one--dimensional and spanned by the $k$--th Hermite polynomial $h_k$. The Wiener chaos decomposition then says that $L^2(d\gamma)$ has $\{h_k\}$ as an orthogonal basis.
This is related to the Wiener chaos decomposition above because polynomials of iterated stochastic integrals are themselves iterated stochastic integrals.

\section{Background on Tracy-Widom distribution}\label{TATW}
The Tracy--Widom distribution \cite{TW} has cumulative distribution function which can be defined as the Fredholm determinant \index{Tracy-Widom distribution}%
$$
F_2(s)=\det(1-A)_{L^2(s)},
$$
where $A$ is the operator on $L^2(s,\infty)$ with kernel
$$
\frac{\Ai(x)\Ai'(y) - \Ai'(x)\Ai(y)}{x-y}.
$$

It can also be defined by
$$
F_2(s) = \exp\left(-\int_s^{\infty} (x-s)q(x)^2dx \right),
$$
where $q(x)$ solves the Painleve II equation
$$
q''(x) = xq(x) + 2q(x)^3
$$
with the boundary condition
$$
q(x) \sim Ai(x), \quad x\rightarrow\infty.
$$
This is called the \textit{Hastings-Mcleod} solution.

\section{Asymptotic analysis}\label{asyana}
The goal is to find asymptotics of expressions of the form
$$
\int g(x) e^{nf(x)}dx
$$
as $n\rightarrow\infty$.

\begin{theorem}
Let $f,g:[a,b]\rightarrow\mathbb{R}$ be smooth. Assume $f$ has a unique global maximum at $c\in (a,b)$ and that $f''(c)<0$. Define $I(n)$ by
$$
\int_a^b g(x)e^{nf(x)}dx = n^{-1/2}e^{nf(c)}I(n).
$$
Then
$$
\lim_{n\rightarrow\infty} I(n) = \sqrt{-\frac{2\pi}{f''(c)}}g(c).
$$
\end{theorem}
\begin{proof} (Sketch)
Since $f(x)$ has a global maximum at $c,$ then as $n\rightarrow\infty$ most of the contributions to the integral come in a small neighbourhood around $c$. Setting $y=n^{1/2}(x-c),$
$$
\int_a^b g(x)e^{nf(x)}dx  = \int_{-n^{1/2}(c-a)}^{n^{1/2}(b-c)} g(c+n^{-1/2}y)e^{nf(c+n^{-1/2}y)}n^{-1/2}dy.
$$
Now use the Taylor expansion $f(c+n^{-1/2}y) \approx f(c) + \tfrac{1}{2}f''(c)n^{-1}y^2$ to get that
$$
\lim_{n\rightarrow\infty} I(n) = g(c) \int_{-\infty}^{\infty} e^{f''(c)y^2/2}dy = \sqrt{-\frac{2\pi}{f''(c)}}g(c).
$$
\end{proof}

Here is an example to find the asymptotics of $n!$. Use that
$$
n! = \Gamma(n+1) = \int_0^{\infty} t^ne^{-t}dt = n^{n+1} \int_0^{\infty} x^ne^{-nx}dx
$$
by the substitution $t=nx$. This is of the right form, with $f(x)=\log x-x,g(x)=1.$ Solving for the maximum of $f$ on $(0,\infty)$ yields $c=1$ with $f(c)=-1,f''(c)=-1$. Therefore
$$
n! \approx n^{n}e^{-n}\sqrt{2\pi n},
$$
which is Stirling's formula.

We will now use similar ideas to find the behaviour of the tails of the Airy function
$$
\Ai(x) = \frac{1}{2\pi } \int_{-\infty}^{\infty} e^{i(t^3/3+xt)}dt.
$$
Note that, as written, the integral is only defined conditionally, since the integrand does not converge to $0$ at $t\rightarrow\pm \infty$. This can be resolved by deforming the contour integral to
$$
\Ai(x) = \frac{1}{2\pi } \int_{\infty\cdot e^{5\pi i/6}}^{\infty\cdot e^{\pi i/6}} e^{i(t^3/3+xt)}dt
$$
Note that the integrand is not yet in the form $e^{xf(t)}$. To resolve this, simply take the substitution $t=x^{1/2}z$ to get
$$
\Ai(x) = \frac{x^{1/2}}{2\pi} \int_{\infty\cdot e^{5\pi i/6}}^{\infty\cdot e^{\pi i/6}} e^{x^{3/2}f(z)}dz
$$
where $f(z)=i(z^3/3+z)$.

As before, we want to find the saddle points of $f(z)$. Solving for $f'(z)=0$, we find that the saddle points are at $t=\pm i $. So we know that we want to deform the contour to go through either $i$ or $-i$ (or both). Furthermore, the absolute value of the integrand should decay quickly as it moves away from the saddle point. Since $\left| e^{x^{3/2}f(z)}\right| = e^{Re(x^{3/2}f(z))}$, we should look for contours where $Re(f(z))<Re(f(i))$ or $Re(f(z)<Re(f(-i))$. The first case is shown in Figure \ref{Positive}.

\begin{figure}[ht]
\caption{Shaded regions indicate where $Re(f(z))<Re(f(i))$.}
\label{Positive}
\begin{center}
\includegraphics[height=6cm]{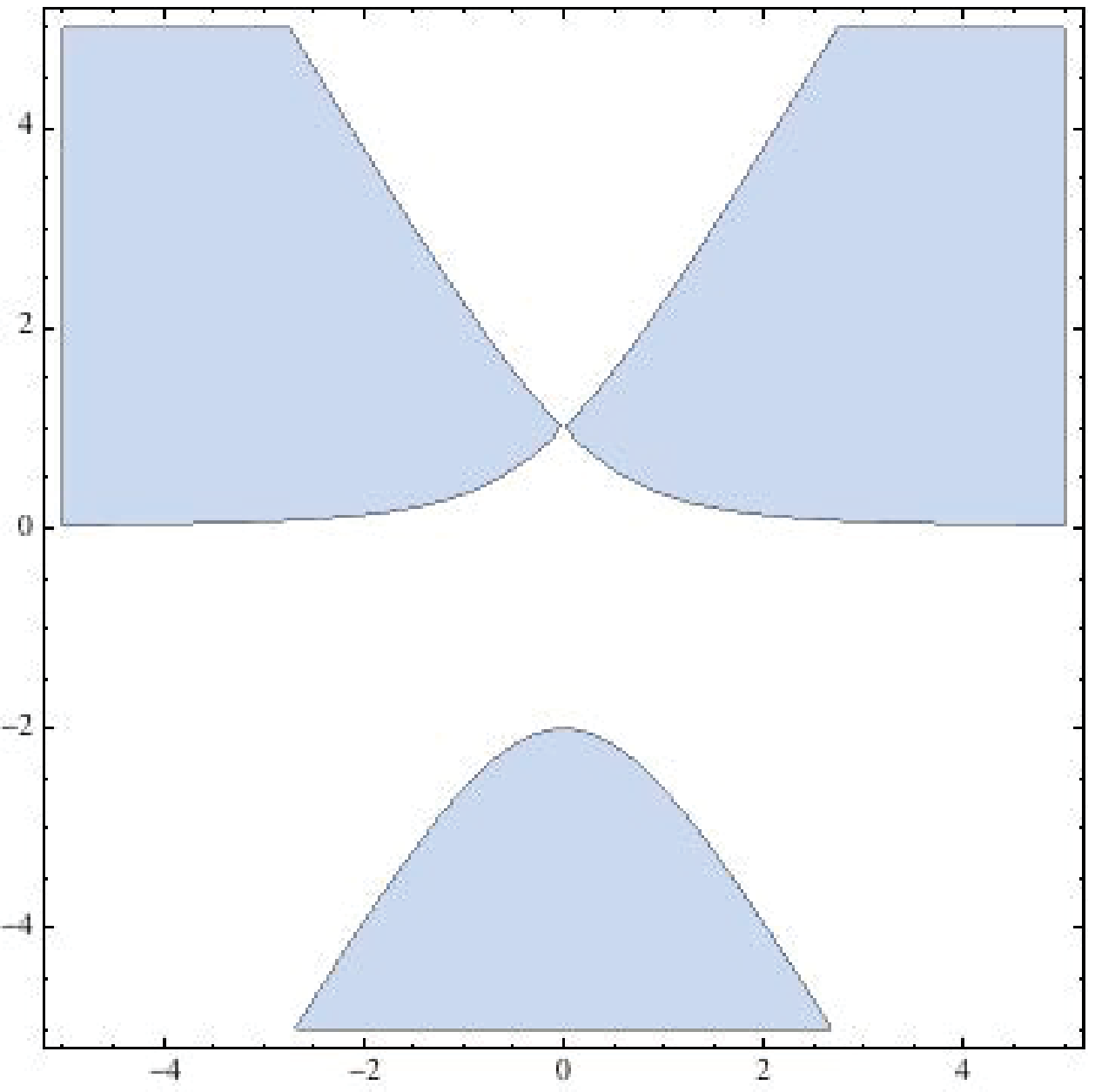}
\end{center}
\end{figure}

Therefore, the contours can be deformed to begin near $\infty e^{5\pi i/6}$, go through $i$ and then end near $\infty e^{i \pi/6}.$ Using the Taylor expansion
$$
f(z) \approx -\frac{2}{3}-(z-i)^2 + \cdots
$$
and the substitution $u=x^{3/4}(z-i)$ we get
$$
\Ai(x) \approx \frac{x^{1/2}e^{-2x^{3/2}/3}}{2\pi} \int e^{-x^{3/2}(z-i)^2}dz \approx \frac{e^{-2x^{3/2}/3}}{2\pi x^{1/4}}  \int_{-\infty}^{\infty} e^{-u^2}du,
$$
This finally yields
$$
\Ai(x) \approx \frac{e^{-\tfrac{2}{3}x^{3/2}}}{2\sqrt{\pi}x^{1/4}}.
$$

\begin{exercise}
Explain what would go wrong if we tried to use the $-i$ saddle point. Figure \ref{Positive2} shows a plot of $Re(f(z))$. Hint: it is important to keep in mind the endpoints of the contour.
\end{exercise}
\begin{figure}[ht]
\caption{Shaded regions indicate where $Re(f(z))<Re(f(-i))$.}
\label{Positive2}
\begin{center}
\includegraphics[height=6cm]{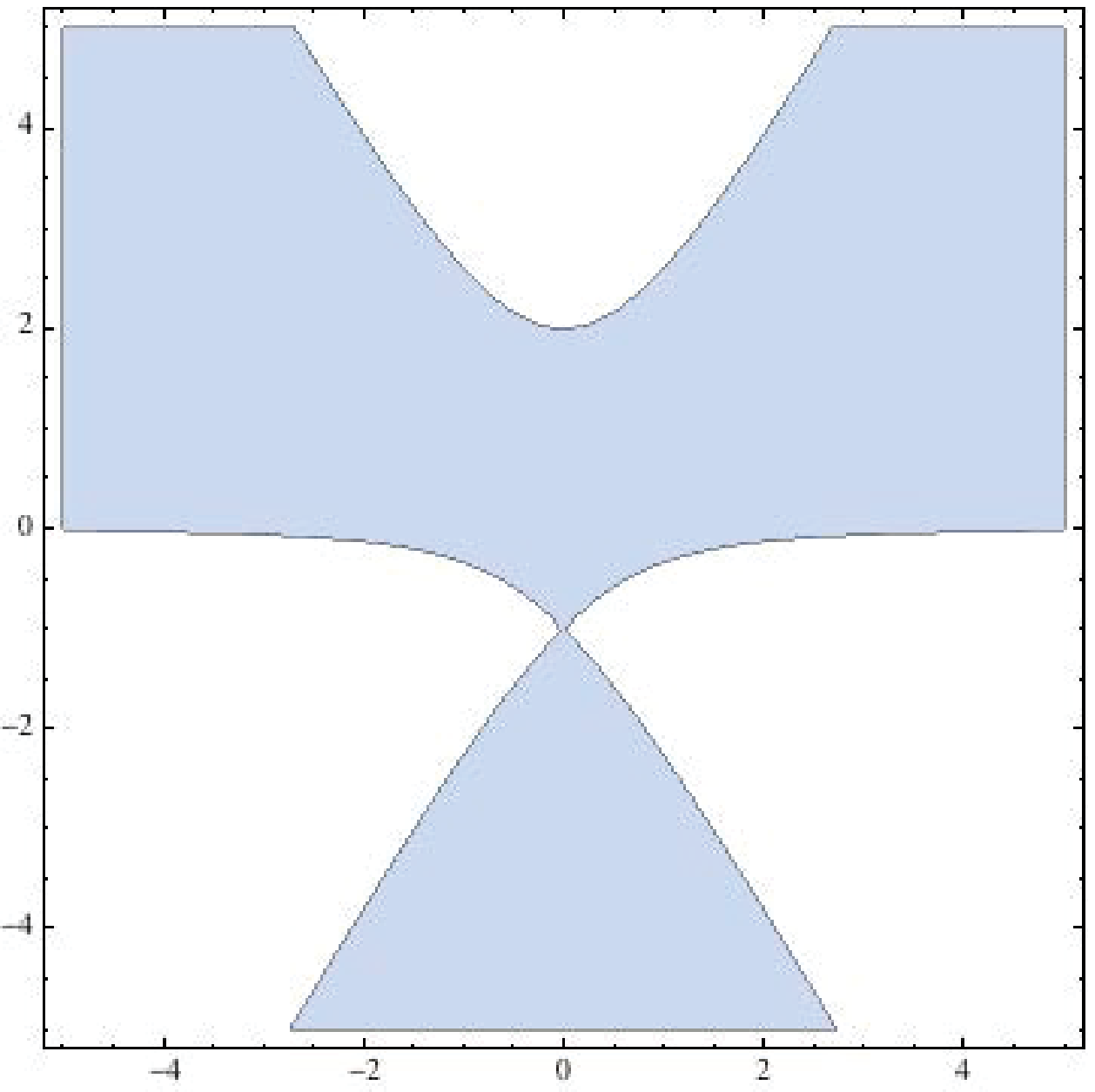}
\end{center}
\end{figure}

\begin{exercise}
Use a similar argument to show that as $x\rightarrow\infty$
$$
\Ai(-x) \approx \frac{1}{\sqrt{\pi}x^{1/4}}\cos\left( \frac{\pi}{4}- \frac{2x^{3/2}}{3}\right).
$$
Figure \ref{Negative} should be helpful.
\end{exercise}

\begin{figure}[ht]
\caption{Shaded regions indicate where $Re(f(z))<Re(f(1))$ for the function  $f(z)=i(z^3/3-z)$.}
\label{Negative}
\begin{center}
\includegraphics[height=6cm,scale=0.75]{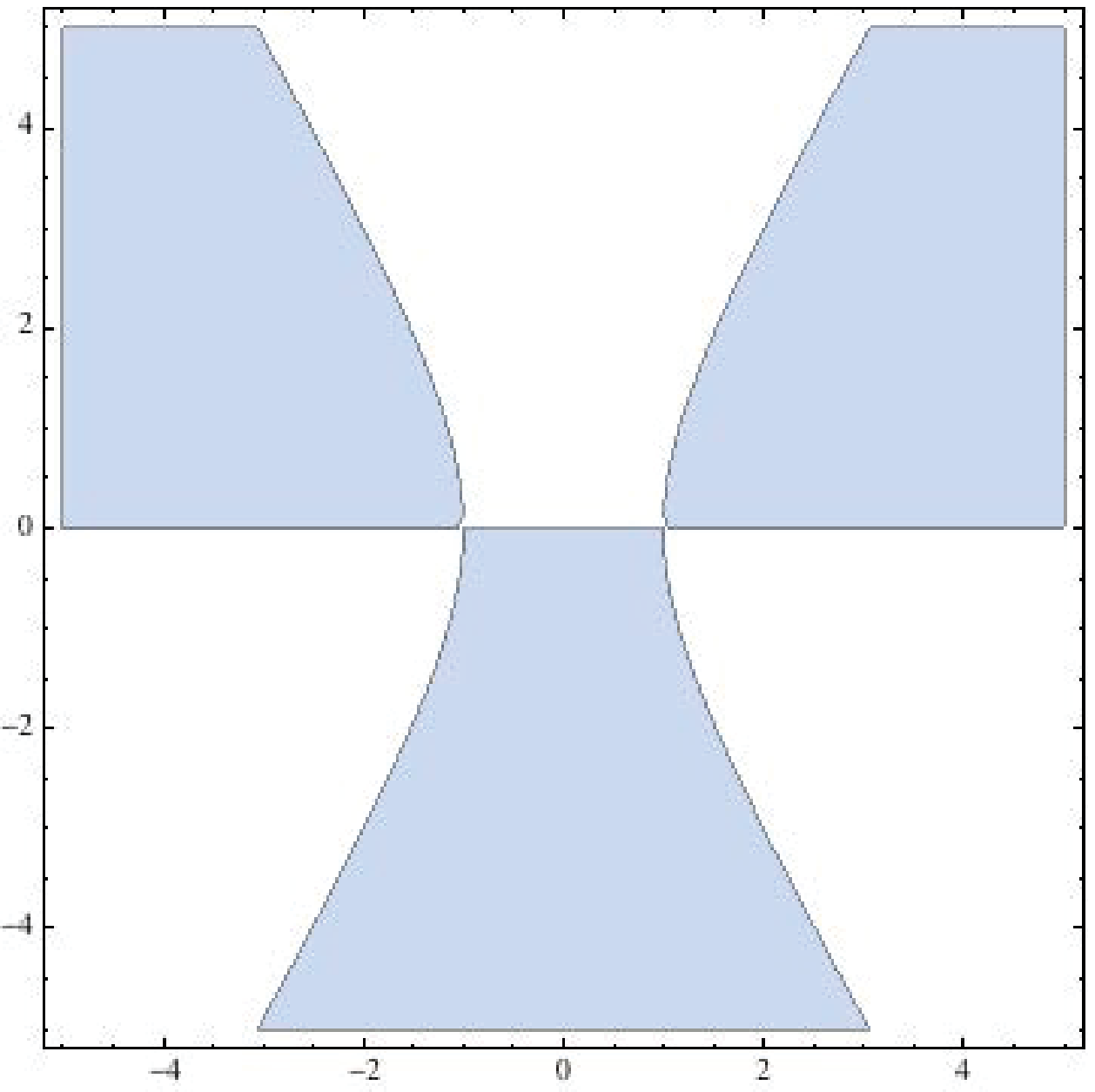}
\end{center}
\end{figure}

\bibliographystyle{amsplain}


\end{document}